\long\def\commentout#1{}

\newif\ifprint
\printfalse
\documentclass[twoside,10pt]{HandbookOfModuli}

\usepackage{amsmath,amsthm}
\usepackage{amssymb}
\usepackage{latexsym}
\usepackage[utf8]{inputenc}
\usepackage{textcomp}
\usepackage{color}
\usepackage{titlesec}
\usepackage{xspace}

\ifprint
	\input{giovanni} 
	\usepackage[final]{microtype}
\fi
\usepackage{bm}
\renewcommand{\mathbf}[1]{\bm{#1}} 

\ifprint
	\definecolor{linkred}{rgb}{0,0,0} 
	\definecolor{linkblue}{rgb}{0,0,0} 
\else
	\definecolor{linkred}{rgb}{0.7,0.2,0.2}
	\definecolor{linkblue}{rgb}{0,0.2,0.6}
\fi

\numberwithin{equation}{section} 

\catcode`@=11 \baselineskip 4.5mm
\def\ps@handbook{\def\@oddhead{\hfill \leftmark \hfill\thepage }
\def\@evenhead{\thepage \hfill \rightmark \hfill}
\def\@oddfoot{}
\def\@evenfoot{}}
\def\@evenhead{}
\def\@oddfoot{}
\def\@evenfoot{\hfill\copyright\ China Higher Education Press}
\def\list#1#2{\ifnum \@listdepth >5\relax \@toodeep \else \global
\advance \@listdepth\@ne \fi \rightmargin \z@ \listparindent\z@
\itemindent\z@ \csname @list\romannumeral\the\@listdepth\endcsname
\def\@itemlabel{#1}\let\makelabel\@mklab \@nmbrlistfalse #2\relax
\@trivlist \parskip -\parsep \parindent\listparindent \advance
\linewidth -\rightmargin \advance\linewidth -\leftmargin \advance
\@totalleftmargin \leftmargin \parshape \@ne \@totalleftmargin
\linewidth \ignorespaces}
\renewcommand*\l@section{\@tocline{1}{0pt}{0em}{1.75em}{}}
\renewcommand*\l@subsection{\@tocline{2}{0pt}{1.75em}{2em}{}} 
\renewcommand\contentsnamefont{\bfseries\large} 
\catcode`@=12
\renewcommand{\theequation}{\thesection.\arabic{equation}}

\def\thebibliography#1{\section*{References}
\list{[\arabic{enumi}]}{\settowidth \labelwidth{[#1]} \leftmargin
\labelwidth \advance \leftmargin \labelsep \usecounter{enumi}}
\def\newblock{\hskip .11em plus .33em minus .07em} \sloppy
\clubpenalty 4000 \widowpenalty 4000 \sfcode`\.=1000 \relax}

\titleformat{\section}{\normalfont\large\bfseries}{\thesection.}{0.5em}{}[\kern0.em]
\titleformat{\subsection}{\normalfont\bfseries}{\thesubsection.}{0.3em}{}[\kern0.em]
\titleformat{\subsubsection}[runin]{\normalfont\bfseries}{\thesubsubsection.}{0.5em}{}[\kern0.5em]

\def\fofsubsubsection#1{\refstepcounter{equation}\subsubsection*{\theequation.\kern0.25em #1}}
\def\foisubsubsection#1{\refstepcounter{equation}\subsubsection*{\kern\parindent\theequation.\kern0.25em #1}}

\usepackage{amscd}
\usepackage{array}

\newcommand{\cA}{\mathcal{A}}

\newcommand{\cG}{\mathcal{G}}

\newcommand{\cM}{\mathcal{M}}

\newcommand{\cV}{\mathcal{V}}

\newcommand{\cX}{\mathcal{X}}
\newcommand\A[1]{\mathcal{A}_{#1}}
\newcommand\M[1]{\mathcal{M}_{#1}}
\newcommand\X[1]{\mathcal{X}_{#1}}
\newcommand\barA[1]{\tilde{\mathcal{A}}_{#1}}
\newcommand\SatA[1]{{\mathcal{A}}^{*}_{#1}}
\newcommand\F[1]{\mathcal{F}_{#1}}
\newcommand{\Xg}{\mathcal{X}_g}
\newcommand{\HDR}{\mathcal{H}^1_{\rm dR}}
\newcommand{\Hg}{\mathfrak{H}_g}

\newcommand{\SatAg}{\mathcal{A}^{*}_g}
\newcommand{\CH}{\mathrm{CH}}
\newcommand{\CHQ}{\mathrm{CH}_{\bQ}}

\newcommand{\bC}{\mathbb{C}}
\newcommand{\bE}{\mathbb{E}}
\newcommand{\bF}{\mathbb{F}}
\newcommand{\bG}{\mathbb{G}}

\newcommand{\bL}{\mathbb{L}}
\newcommand{\bP}{\mathbb{P}}
\newcommand{\bQ}{\mathbb{Q}}
\newcommand{\bR}{\mathbb{R}}

\newcommand{\bZ}{\mathbb{Z}}

\newcommand{\GL}{\mathrm{GL}}
\newcommand{\GSp}{\mathrm{GSp}}

\newcommand{\mult}{\mathrm{m}}

\newcommand{\Sp}{\mathrm{Sp}}
\newcommand{\SpgZ}{\mathrm{Sp}(2g,\bZ)}

\newcommand{\Gm}{\bG_\mult}

\newcount\caseno\caseno=0
\def\nr{\global\advance\caseno by 1({\number\caseno})}

\usepackage{url}

\newtheorem{theorem}[equation]{Theorem}
\newtheorem{proposition}[equation]{Proposition}
\newtheorem{lemma}[equation]{Lemma}
\newtheorem{corollary}[equation]{Corollary}
\newtheorem{conjecture}[equation]{Conjecture}

\theoremstyle{definition}
\newtheorem{definition}[equation]{Definition}
\newtheorem{example}[equation]{Example}

\numberwithin{equation}{section}

\begin{document}
\setcounter{page}{1}


\title{The Cohomology of the Moduli Space of Abelian Varieties}
   
\author{Gerard van der Geer}
\address{University of Amsterdam}
\email{geer@science.uva.nl}

   

\maketitle
\thispagestyle{empty}


\section*{Introduction}
That the moduli spaces of abelian varieties are a rich source of 
arithmetic and geometric information slowly emerged
from the work of Kronecker, Klein, Fricke and many others at the end 
of the 19th century.
Since the 20th century we know that the first place to dig for these
hidden treasures is the cohomology of these moduli spaces. 

In this survey we are concerned with the cohomology of the moduli
space of abelian varieties. Since this is an extensive and widely
ramified topic that connects to many branches of algebraic geometry
and number theory we will have to limit ourselves. I have chosen to
stick to the moduli spaces of principally polarized abelian varieties,
leaving aside the moduli spaces of abelian varieties with non-principal
polarizations and the other variations like the moduli spaces with extra
structure (like conditions on their endomorphism rings) since often
the principles are the same, but the variations are clad in
heavy notation.

The emphasis in this survey is on the tautological ring of the moduli
space of principally polarized abelian varieties. We discuss the cycle
classes of the Ekedahl-Oort stratification, that can be expressed in
tautological classes, and discuss differential forms on the moduli
space. We also discuss complete subvarieties of~$\A{g}$.
Finally, we discuss Siegel modular forms and its relations
to the cohomology of these moduli spaces. We sketch the approach developed
jointly with Faber and Bergstr\"om to calculate the traces of the Hecke
operators by counting curves of genus $\leq 3$ over finite fields,
an approach that opens a new window on Siegel modular forms.

\setcounter{tocdepth}{1}
\tableofcontents

\section{The Moduli Space of Principally Polarized Abelian Varieties}\label{Ag}

We shall assume the existence of the moduli space of principally
polarized abelian varieties as given. So
throughout this survey $\A{g}$ will denote the Deligne-Mumford stack
of principally polarized abelian varieties of dimension $g$. It is 
a smooth Deligne-Mumford stack over ${\rm Spec}({\bZ})$ of relative
dimension $g(g+1)/2$, see \cite{F-C}.

Over the complex numbers this moduli space can be described as the
arithmetic quotient (orbifold)
 ${\Sp}(2g,{\bZ}) \backslash {\Hg}$
of the Siegel upper half space by the symplectic group.
This generalizes the well-known
description of the moduli of complex elliptic curves as
${\rm SL}(2,{\bZ})\backslash \mathfrak{H}$ with $\mathfrak{H}=\mathfrak{H}_1$
the usual upper half plane of the complex plane. We refer to Milne's
account in this Handbook for the general theory of Shimura varieties.

The stack $\A{g}$ comes with a universal family of principally 
polarized abelian varieties $\pi: \X{g} \to \A{g}$. Since abelian 
varieties can degenerate the stack $\A{g}$ is not proper or complete.

The moduli space $\A{g}$ admits several compactifications. The first one
is the Satake compactification or Baily-Borel compactification. It is 
defined by considering the vector space of Siegel modular forms of 
sufficiently high weight, by using these to map $\A{g}$ to projective
space and then by taking the closure of the image 
of $\A{g}$ in the receiving projective space. This construction was
first done by Satake and by Baily-Borel over the field of
complex numbers, cf.\ \cite{B-B}. 
The Satake compactification $\SatA{g}$ is very singular for $g\geq 2$. 
It has a stratification
$$
\SatA{g}=\A{g}\sqcup \SatA{g-1}=\A{g}\sqcup \A{g-1} \sqcup \cdots \sqcup
\A{1}\sqcup \A{0}.
$$
In an attempt to construct non-singular compactifications of arithmetic
quotients, such  as ${\Sp}(2g,{\bZ}) \backslash {\Hg}$, Mumford with a team of 
co-workers created a theory of so-called toroidal compactifications of 
${\A{g}}$ in \cite{AMRT}, cf.\ also the new edition \cite{AMRT-2}. 
These compactifications are not unique but depend on a 
choice of combinatorial data, an admissible cone decomposition of the cone of
positive (semi)-definite symmetric bilinear forms in $g$ variables. 
Each such toroidal compactification
admits a morphism $q: \barA{g} \to \SatA{g}$. 
Faltings and Chai showed how Mumford's theory of compactifying quotients of
bounded symmetric domains could be used to extend
this to a compactification of $\A{g}$ over the integers, 
see \cite{F1, F-C, AMRT}. 
This also led to the Satake compactification over the integers.

There are a number of special  choices for the cone decompositions, 
such as the second 
Voronoi decomposition, the central cone decomposition or the 
perfect cone decomposition. Each of these choices has its advantages 
and disadvantages. Moreover, Alexeev has constructed a functorial 
compactification of $\A{g}$, see \cite{Alexeev, Alexeev-Nakamura}. 
It has as a disadvantage that for $g\geq 4$ 
it is not irreducible but possesses extra components. 
The main component of this  corresponds to
${\cA}_g^{\rm Vor}$, the toroidal compactification 
defined by the second Voronoi compactification. 
Olsson has adapted Alexeev's  construction using log  structures 
and obtained a functorial compactification. The normalization 
of this compactification is the compactification ${\cA}_g^{\rm Vor}$, cf\ 
\cite{Olsson}.

The toroidal compactification ${\cA}_g^{\rm perf}$ defined by the perfect cone
decomposition is a canonical model of $\A{g}$ for $g \geq 12$
as was shown by Shepherd-Barron, see \cite{Sh-B} and his contribution to 
this Handbook. 

The partial desingularization of the Satake compactification obtained by Igusa
in \cite{Igusa} 
coincides with toroidal compactification corresponding to the central cone
decomposition. 
We refer to a survey
by Grushevsky (\cite{Gr}) on the geometry of the moduli space
of abelian varieties.

These compactifications (2nd Voronoi, central and perfect cone)
agree for $g\leq 3$, but are different for higher $g$.
For $g=1$ one has $\barA{1}=\SatA{1}=\overline{\cM}_{1,1}$, the Deligne-Mumford
moduli space of $1$-pointed stable curves of genus $1$. For $g=2$ 
the Torelli morphism gives an identification $\barA{2}=\overline{\cM}_{2}$, 
the Deligne-Mumford moduli space of stable curves of genus $2$. 
The open part $\A{2}$ corresponds to stable curves of genus $2$ of
compact type. But please note that the Torelli 
morphism of stacks ${\cM}_{g} \to \A{g}$
is a morphism of degree $2$ for $g\geq 3$, since every 
principally polarized abelian variety
posseses an automorphism of order $2$, 
but the generic curve of genus $g\geq 3$
does not.

We shall use the term Faltings-Chai compactifications for the compactifications
(over rings of integers) defined by admissible cone decompositions.
 
\section{The Compact Dual}\label{compactdual}
The moduli space $\A{g}(\bC)$ has the analytic description as ${\Sp}(2g,{\bZ})
\backslash {\Hg}$. The Siegel upper half space ${\Hg}$
can be realized in various ways,
one of which is as an open subset of the so-called compact dual and the
arithmetic quotient inherits various properties from this compact quotient.
For this reason we first treat the compact dual at some length.

To construct $\mathfrak{H}_g$, 
we start with an non-degenerate symplectic form on a complex
vector space, necessarily of even dimension, say $2g$. To be explicit, consider
the vector space ${\bQ}^{2g}$ with basis $e_1,\ldots,e_{2g}$ and 
symplectic form $\langle \, , \, \rangle $
given by $J(x,y)=x^t J y$ with 
$$
J= \left( \begin{matrix} 0 & 1_g \\ -1_g & 0 \\ \end{matrix} \right) \, .
$$
We let $G={\rm GSp}(2g,{\bQ})$ be the algebraic group over 
${\bQ}$ of symplectic similitudes of this symplectic vector space
$$
G= \{ g \in {\GL}(2g,{\bQ}): J(gx,gy)= \eta(g) J(x,y)\} \, ,
$$
where $\eta: G \to {\Gm}$ is the so-called multiplyer. 
Then $G$ is the group of matrices
$\gamma=(A \, B; C \, D)$ with $A,B,C,D$ integral $g \times g$-matrices
with 
$$
A^t \cdot C = C^t \cdot A, \, B^t \cdot D = D^t \cdot B \quad
\hbox{\rm and} \quad A^t \cdot D - C^t \cdot B = \eta(g) 1_{2n}.
$$
We denote the kernel of $\eta$ by $G^0$.

Let $Y_g$ be the Lagrangian Grassmannian
$$
Y_g =\{ L \subset {\bC}^{2g}: \dim(L)=g,\, 
 J(x,y)=0 \, \text{for all $ x,y \in L$}\}
$$
that parametrizes all totally isotropic subspaces of our complex symplectic 
vector space. This is a homogeneous manifold of complex dimension $g(g+1)/2$
for the action of $G({\bC})={\GSp}(2g,{\bC})$; 
in fact this group acts transitively
and the quotient of ${\GSp}(2g,{\bC})$ by the central ${\Gm}({\bC})$ 
acts effectively. 
We can write $Y_g$ as a quotient $Y_g=G({\bC})/Q$, where $Q$ is a 
parabolic subgroup of $G({\bC})$. More precisely, if we fix a point
$y_0= e_{1}\wedge \ldots \wedge e_{g} \in Y_g$
then we can write $Q$ as the group of matrices $(A\, B ; C \, D)$ in
${\GSp}(2g,{\bC})\}$ with $C=0$. This parabolic group $Q$ has a Levi 
decomposition as $Q=M \ltimes U$ with $M$ the subgroup of $G$
that respects the decomposition of the symplectic space as 
${\bQ}^g\oplus {\bQ}^g$; the matrices of $M$ are of the form 
$(A \, 0 ; 0 \, D)$ and is isomorphic to ${\rm GL}(g)\times {\bG}_m$,
while those of $U$ are of the form $(1_g \, B ; 0 \, 1_g)$ with
$B$ symmetric.  

There is an embedding of ${\Hg}$ into $Y_g$ as follows.
We consider the group $G^0({\bR})$
and the maximal compact subgroup $K$ of elements that fix $\sqrt{-1} \, 1_g$.
Its elements are described as $(A\, -B ; B \, A)$ and assigning to
it the element $A+\sqrt{-1} B$ gives an isomorphism of $K$ with
the unitary group  $U(g)$. One way to describe the Siegel upper half space
is as the orbit $X_g=G^0({\bR})/K$ under the action of $G^0$; this can be 
embedded in $Y_g=(G/Q)({\bC})$ as the set of all maximal
isotropic subspaces $V$ 
such that $-\sqrt{-1} \langle v,\bar{v} \rangle $ is positive definite
on $V$. Each such subspace has a basis consisting of the columns of the
transpose of the matrix $(-1_g \, \tau)$ for a unique $\tau \in {\Hg}$.
The subgroup $G^{+}({\bR})$ leaves this subset invariant and
this establishes the embedding of the domain ${\Hg}$ in its $Y_g$
and this space $Y_g$ is called the {\sl compact dual} of ${\Hg}$.
It contains $X_g \thicksim {\Hg}$ as
an open subset. The standard example (for $g=1$) is that of the 
upper half plane contained in ${\bP}^1$.

For later use we extend this a bit by looking not only at 
maximal isotropic subspaces, but also at symplectic filtrations
on our standard symplectic space.

Consider for $i=1, \ldots ,g$ the (partial) flag variety $U_g^{(i)}$ of
symplectic flags $E_i \subset \ldots \subset  E_{g-1} \subset E_g$,
of linear subspaces $E_j$ of ${\bC}^{2g}$ with  $\dim (E_j)=j$ and $E_g$  
totally isotropic. We have $U_g^{(g)} = Y_g$. There are  natural maps 
$\pi_i: U_g^{(i)} \to U_g^{(i+1)}$ and the fibre of $\pi_i$ is a 
Grassmann variety of dimension $i$. We can represent $U_g^{(1)}$ as a 
quotient $G/B$, where $B$ is a Borel subgroup of $G$. 
These spaces $U_g^{(i)}$ come equipped with universal flags
$E_i \subset E_{i+1} \subset \ldots \subset E_g$.

The manifold $Y_g$ possesses, as all Grassmannians 
do, a cell decomposition. To define it, choose a fixed
isotropic flag
$
\{ 0 \} = Z_0 \subsetneq Z_{1} \subsetneq Z_2 \subsetneq
\ldots \subsetneq  Z_g$; that is, $\dim Z_i=i$ and $Z_g$ is an isotropic
subspace of our symplectic space.
We extend the filtration by setting
 $Z_{g+i} = (Z_{g-i})^{\bot}$ for $i=1,\ldots, g$.
For general $V \in Y_g$ we expect that  $V\cap Z_{j} = \{ 0 \} $ for
$j\leq g$. Therefore,
for $\mu=(\mu_1,\ldots, \mu_r)$ with $\mu_i$ non-negative
integers satisfying
$$
0 \leq \mu_i \leq g,\quad
\mu_i-1\leq \mu_{i+1} \leq \mu_{i} \eqno(1)
$$
we put 
$$
W_{\mu}=\{ V \in Y_g : \dim (V \cap Z_{g+1-\mu_i})=i\}\,  .
$$
This gives a cell decomposition with cells only in even real dimension. The cell
$W_{\mu}$ has (complex) codimension $\sum \mu_i$.
Denote the set of  $n$-tuples $\mu = (\mu_1,\ldots , \mu_g)$
satisfying (1) by $M_g$. Then $\# M_g = 2^g$.  Moreover, we
have
$$
W_{\mu} \subseteq {\overline W}_{\nu}\iff \mu_i\geq
\nu_i \quad {\rm for }\quad 1\leq i \leq g \, .
$$
From the cell decomposition we find the homology of $Y_g$.
\begin{proposition}\label{cohcompdual}
The integral homology of $Y_g$ is generated by the cycle classes of
$[{\overline W}_{\mu}]$ of the closed cells with $\mu \in M_g$. 
The Poincar\'e-polynomial of
$Y_g$ is given by $\sum b_{2i} \, t^i=(1+t)(1+t^2)\cdots
(1+t^g)$.
\end{proposition}

Note that $Y_g$ is a rational variety and that the Chow ring 
$R_g$ of $Y_g$ is the isomorphic image of the cohomology ring 
of $Y_g$ under the usual cycle class map. On $Y_g$ we
have a sequence of tautological vector bundles 
$0 \to E \to H \to Q \to 0$,
where $H$ is the trivial bundle of rank $2g$ defined by our fixed 
symplectic space and where the fibre of $E_y$ of $E$ over $y$ is 
the isotropic subspace of dimension $g$ corresponding to $y$. 
The bundle $E$ corresponds to the standard representation of $K=U(g)$, 
see also Section \ref{The Hodge Bundle}. 
The tangent space to a point $e=[E]$ of $Y_g$ is ${\rm
Hom}^{\rm sym}(E,Q)$, the space of symmetric homomorphisms from
$E$ to $Q$; indeed, usually the tangent space of a Grassmannian is
described as ${\rm Hom}(E,Q)$ (``move $E$ a bit and it moves infinitesimally
out of the kernel of $H\to Q$"), but we have to preserve the symplectic
form that identifies $E$ with the dual of $Q$; therefore we have
to take the `symmetric' homomorphisms. 
We can identify this with ${\rm Sym}^2(E^{\vee})$.

We consider the Chern classes  $u_i= c_i(E) \in R_g= {\CH}^*(Y_g)$ for $
i=1,\ldots,g$. We call them {\sl tautological classes}. The symplectic form
$J$ on $H$ can be used to identify $E$ with the dual of the
quotient bundle $Q$ . The triviality of the bundle $H$ 
implies the following relation for the Chern classes of $E$ in $R_g$:
$$
(1+u_1+u_2+\ldots +u_g)(1-u_1+u_2+\ldots +(-1)^gu_g)=1 \, .
$$
These relations may be succinctly stated as
$${\rm ch}_{2k}(E)=0 \qquad (k\geq 1), \eqno(2)$$
where ${\rm ch}_{2k}$ is the part of degree $2k$ of the Chern character.

\begin{proposition} The Chow ring $R_g$ of $Y_g$ is
the ring generated by the $u_i$ with as only relations
${\rm ch}_{2k}(E)=0$ for $k\geq 1$.
\end{proposition}

One can check algebraically that the ring which is the quotient of
${\bZ}[u_1,\ldots,u_g]$ by the relation (2) has Betti numbers as
given in Prop.\ \ref{cohcompdual}.  This description  implies that
this ring after tensoring with ${\bQ}$ 
is in fact generated by the $u_{j}$ with $j$ odd.
Furthermore, by using induction one obtains the relations
$$
u_g u_{g-1}\cdots
u_{k+1}u_k^2=0  \qquad \text{for $k=g, \ldots, 1$}. \eqno(3)
$$
It follows that this ring has the following set of $2^g$  
basis elements
$$
\prod_{i \in I} u_i \, ,\qquad I \subseteq \{1,\ldots,g\}.
$$
The ring $R_g$ is a so-called Gorenstein ring with socle 
$u_gu_{g-1} \cdots u_1$.

Using $R_g/ (u_g) \cong R_{g-1}$ one finds the following properties.
\begin{lemma} 
In $R_g/(u_g)$ we have $u_1^{g(g-1)/2}\neq 0$ and $u_1^{g(g-1)/2+1}=0$. 
\end{lemma}

Define now classes in the Chow ring of the flag spaces
$U_g^{(i)}$ as follows:
$$
v_j = c_1(E_j/E_{j-1}) \quad  \in CH^*(U_g^{(i)})\quad j=i, \ldots, g.
$$
Moreover, we set
$$
u_j^{(i)} = c_j(E_i) \quad \in CH^*(U_g^{(i)}) \quad j=1,\ldots, i.
$$
We can view $u_j^{(i)}$ as the $j$th symmetric function in $v_1, \ldots,
v_i$. Then we have the relations under the forgetful maps $\pi_i:
U_g^{(i)} \to U_g^{(i+1)}$
$$
\pi_i^*(u_j^{(i+1)})= u_j^{(i)} + v_{i+1} u_{j-1}^{(i)}\qquad j=1,\ldots,
i+1,\eqno(4)
$$
and
$$
v_j^j-u_1^{(j)}v_j^{j-1}+ \ldots + (-1)^ju_j^{(j)}=0 \qquad j=1,\ldots,
g.\eqno(5) $$
The Chow ring of $U_g^{(1)}$ has generators
$$
v_1^{\eta_1}v_2^{\eta_2} \cdots v_g^{n_g}
u_g^{\epsilon_g}u_{g-1}^{\epsilon_{g-1}}\cdots
u_1^{\epsilon_1}
$$
with $0\leq \eta_i < i$ and $\epsilon_i =0 $ or $1$.
\begin{lemma}\label{Gysin} 
The following Gysin formulas hold:
\begin{enumerate}
\item{} $(\pi_i)_*(c)=0 {\rm ~for~all~} c \in CH^*(U_g^{(i+1)})$;
\item{} $(\pi_i)_*(v_{i+1}^i)=1$;
\item{} $(\pi_i)_*(u_i^{(i)})= (-1)^i$.
\end{enumerate}
\end{lemma}
These formulas together with (4) and (5) completely determine the image of the
Gysin map $(\pi)_*= (\pi_{g-1} \cdots \pi_1)_*$.
\bigskip

For a sequence of $r\leq g$ positive integers $\mu_i$
with $\mu_i \geq \mu_{i+1}$ we define an element of $R_g$:
$$
\Delta_{\mu}(u)= \det \left( \begin{matrix}
(u_{\mu_i-i+j})_{1\leq i \leq r; 1\leq j \leq
r} \end{matrix} \right).
$$
We also define so-called $Q$-polynomials by the formula:
$$
Q_{i,j} = u_i u_j - 2 u_{i+1}u_{j-1} + \ldots
+(-1)^j 2 u_{i+j} ~~{\rm for}~~ i<j.
$$
We have $Q_{i 0} = u_i$ for $i=1, \ldots ,g$.
Let $\mu$ be a strict partition.
For $r$ even we define an anti-symmetric matrix $[x_{\mu}]
=[x_{ij}]$ as follows. Let
$$
x_{i,j} = Q_{\mu_i \mu_j}(u)
\quad {\rm for ~~} 1\leq i < j \leq r.
$$
 We then set for even $r$
$$
\Xi_{\mu} = {\rm Pf}([x_{\mu}]),
$$
while for $r$ odd we define $\Xi_{\mu} = \Xi_{\mu_1 \ldots \mu_r 0}$.
These expressions may look a bit artificial, but their purpose is 
clearly demonstrated by the following Theorem, due to Pragacz \cite{Pragacz,F-P}.

\begin{theorem} (Pragacz's formula)   The
class of the cycle $[{\overline W}_{\mu}]$ in the Chow ring
is given by (a multiple of) $\Xi_{\mu}(u)$.
\end{theorem}
From (2.4) we have the property:
for partitions $\mu$ and $\nu$
with $\sum \mu_i + \sum \nu_i = g(g+1)/2$ we have
$$
\Xi_{\mu} \Xi_{\nu} = 
\begin{cases}   1 & \text{if $\nu = \rho -\mu $,} \\
 0 & \text{if $\nu \neq \rho -\mu$,} \\ 
\end{cases}
$$
where $\rho = \{ g, g-1, \ldots, 1 \}$.
 We have the following relations in $CH^*(Y_g)$:
$$
\begin{aligned}
i) \quad & u_gu_{g-1}\ldots u_1=1,\cr
ii)\quad & u_1^N= N!\prod_{k=1}^g \frac{1}{(2k-1)!!} \, , 
\end{aligned}
$$
where $1$ represents the class of a point and $N=g(g+1)/2$.
\begin{proof}
For i) We have $T_{Y_g}\cong {\rm Sym}^2(E^{\vee})$, thus  the
Chern classes of the tangent bundle $T_{Y_g}$  are
expressed in the $u_i$. By \cite{Fulton}, Ex.\ 14.5.1, p.\ 265)
for the top Chern class of
${\rm Sym}^2$ of a vector bundle we have
$$
e(Y_g)= 2^g = 2^g \Delta_{(g, g-1, \ldots , 1)}(u)= 2^g \det \left(
\begin{matrix} u_g&0&0&\dots &0\cr
u_{g-2}&u_{g-1}&u_g&0\ldots\cr
&&\dots\cr
&&&&u_1\cr
\end{matrix}
\right).
$$
Developing the determinant gives
$$
1= u_gu_{g-1}\ldots
u_1+u_g^2A(u)+u_gu_{g-1}^2B(u)+ \ldots \, ,
$$
from which the desired identity results by using (3).
Alternatively, one may use Pragacz's formula above for the degeneracy 
locus given by $Z_g = V$ for some subspace $V$. 
\end{proof}
\section{The Hodge Bundle}\label{The Hodge Bundle}
Since $\A{g}$ comes with a universal principally polarized abelian variety
$\pi: \X{g} \to \A{g}$ we have a rank $g$ vector bundle or locally
free sheaf 
$$
{\bE}={\bE}_g = \pi_*(\Omega^1_{\X{g}/\A{g}}),
$$
called the {\sl Hodge Bundle}.
It has an alternative definition as 
$$
{\bE}= s^*(\Omega^1_{\X{g}/\A{g}})\, ,
$$
i.e., as the cotangent bundle to the zero section $s: \A{g} \to \X{g}$.
If ${\cX}_g^t$ is the dual abelian variety (isomorphic to $\X{g}$ because
we stick to a principal polarization) then the Hodge bundle ${\bE}^t$
of ${\cX}_g^t$ satisfies
$$
({\bE}^t)^{\vee}={\rm Lie}({\cX}_g^t)\cong R^1\pi_* {\mathcal O}_{\X{g}}.
$$
The Hodge bundle ${\bE}$ can be extended to any toroidal compactification
$\barA{g}$ of Faltings-Chai type. In fact, over $\barA{g}$ we have a
universal family of semi-abelian varieties and one takes the dual
of ${\rm Lie}(\tilde{\cX})$, cf.\ \cite{F-C}. 
We shall denote it again by ${\bE}$.

If we now go to a fine moduli space, say $\A{g}[n]$ with $n\geq 3$,
the moduli space of principally polarized abelian varieties with a
level $n$ structure, and take a smooth toroidal compactification
then we can describe the sheaf of holomorphic 
$1$-forms in terms of the Hodge bundle. With
 $D=\barA{g}[n]-\A{g}[n]$, the divisor at infinity,
we have the important result:

\begin{proposition}\label{OmegaisSym2}
The Hodge bundle ${\bE}$ on $\barA{g}[n]$ for $n\geq 3$
satisfies the identity
$$
{\rm Sym}^2({\bE})\cong \Omega^1_{\barA{g}[n]}(\log D).
$$
\end{proposition}

This result extends the description of the tangent space to ${\Hg}$
in the compact dual. It is proven in the general setting 
in \cite{F-C}, p.\ 86.

Recall the description of $Y_g$ as the symplectic Grassmannian 
$G({\bC})/Q({\bC})$ with $G={\rm GSp}(2g,{\bZ})$ and $Q$ the parabolic subgroup
with Levi decomposition $M\ltimes U$ with $M={\rm GL}(g)\times {\bG}_m$. 
The Siegel upper half space $\mathfrak{H}_g$
can be viewed as an open subset of $Y_g$. Put $G_0={\rm Sp}(2g,{\bZ})$,
$Q_0=Q\cap G_0$ and $M_0=M \cap G_0$. Then if $\rho: Q_0 \to {\rm GL}(V)$
is a finite-dimensional complex representation, we define an equivariant
vector bundle ${\cV}_{\rho}$ on $Y_g$ by 
$$
{\cV}_{\rho} = G_0({\bC}) \times^{Q_0({\bC})} V,
$$
where the contracted product is defined by the usual equivalence relation
$(g,v) \sim (gq, \rho(q)^{-1} v)$ for all $g \in G_0({\bC})$ and 
$q \in Q_0({\bC})$. Then our group ${\rm Sp}(2g, {\bZ})$ acts
on the bundle ${\cV}_{\rho}$ and the quotient is a vector bundle $V_{\rho}$
in the orbifold sense on 
$\A{g}({\bC})={\rm Sp}(2g, {\bZ})\backslash \mathfrak{H}_g$.

A representation of ${\rm GL}(g)$ can be lifted to a representation 
of $Q_0$ by letting it act trivially on the unipotent radical $U$ 
of $M_0$. Carrying out this construction 
with the standard (tautological) representation 
of ${\rm GL}(g)$ produces the Hodge bundle. 

The Hodge bundle can be extended to a bundle over our toroidal 
compactification $\barA{g}$. Since any bundle $V_{\rho}$ is obtained
by applying Schur functors to powers of the Hodge bundle 
(see \cite{F-H}) we can extend
the bundle $V_{\rho}$ by applying the Schur functor to the extended
power of the Hodge bundle. In this way we obtain a canonical extension
to $\tilde{\A{g}}$ for all equivariant holomorphic bundles $V_{\rho}$.

\section{The Tautological Ring of $\A{g}$}
The moduli space $\A{g}$ is a Deligne-Mumford stack or orbifold
and as such it has a Chow ring. To be precise, consider the moduli 
space $\A{g}\otimes k$ over an algebraically closed field $k$. 
For simplicity we shall write $\A{g}$ instead. 

We are interested in
the Chow rings $\CH(\A{g})$ and ${\CHQ}(\A{g})$ of this moduli space.
In general these rings seems unattainable. However, they contain subrings 
that we can describe well and that play an important role.

We denote the Chern classes of the Hodge bundle by
$$
\lambda_i := c_i ({\bE}), \qquad i=1,\ldots,g.
$$
We define the {\sl tautological subring} of the Chow ring 
${\CHQ}(\A{g})$ as the subring generated by the Chern classes of the Hodge
bundle ${\bE}$.

The main result is a description of the tautological ring in terms of
the Chow ring $R_{g-1}$ of the compact dual $Y_{g-1}$ (see \cite{vdG:Cycles}).

\begin{theorem}\label{tautringofAg}
The tautological subring $T_g$ of the Chow ring ${\CHQ}(\A{g})$ generated 
by the $\lambda_i$ is isomorphic to the ring $R_{g-1}$.
\end{theorem}

This implies that a basis of the codimension $i$ part is given by the monomials
$$
\lambda_1^{e_1} \lambda_2^{e_2} \cdots \lambda_{g-1}^{e_{g-1}}
\quad \text{ $e_j \in \{0,1\}$ and   $\sum_{j=1}^{g-1} j\, e_j =i$.}
$$
This theorem follows from the following four results, each interesting
in its own right.

\begin{theorem}\label{lambdarelation}
The Chern classes $\lambda_i$ of the Hodge bundle ${\bE}$ satisfy the relation
$$
(1+\lambda_1+\cdots +\lambda_g)(1-\lambda_1+\cdots +(-1)^g \lambda_g)=1.
$$
\end{theorem} 

\begin{theorem}\label{lambdagiszero}
The top Chern class $\lambda_g$ of the Hodge bundle ${\bE}$ vanishes 
in ${\CHQ}(\A{g})$.
\end{theorem}

\begin{theorem}\label{lambda1isample}
The first Chern class $\lambda_1$ of ${\bE}$ is ample on $\A{g}$.
\end{theorem}

\begin{theorem}\label{completecodimg}
In characteristic $p>0$ the moduli space $\A{g}\otimes {\bF}_p$ contains
a complete subvariety of codimension $g$.
\end{theorem}

To deduce Theorem \ref{tautringofAg} from Thms \ref{lambdarelation} -
\ref{completecodimg} we argue as follows.
By Theorem \ref{lambdarelation} the tautological subring $T_g$ is 
a quotient ring of $R_g$ via $u_i \mapsto \lambda_i$, 
and then by Theorem \ref{lambdagiszero} also
of $R_{g-1}\cong R_g/(u_g)$. The ring $R_{g-1}$ is a Gorenstein ring with 
socle $u_1u_2 \cdots u_{g-1}$ and this is a non-zero multiple of 
$u_1^{g(g-1)/2}$. By Theorem \ref{lambda1isample} the ample class
$\lambda_1$ satisfies $\lambda_1^j \cdot [V]\neq 0$ on any complete
subvariety $V$ of codimension~$j$ in $\A{g}$. The existence of a complete
subvariety of codimension $g$ in positive characteristic implies that
$\lambda_1^{g(g-1)/2}$ does not vanish in ${\CHQ}(\A{g}\otimes {\bF}_p)$,
hence the class $\lambda_1^{g(g-1)/2}$ does not vanish in ${\CHQ}(\A{g})$.
If the map $T_g \to R_{g-1}$ defined by $\lambda_i \mapsto u_i$ would 
have a non-trivial kernel then $\lambda_1^{g(g-1)/2}$ would have to vanish.

Theorem \ref{lambdarelation} was proved in \cite{vdG:Cycles}.
The proof is obtained by
applying the Grothen\-dieck-Riemann-Roch formula to the theta divisor
on the universal abelian variety ${\Xg}$ over $\A{g}$. In fact, take
a line bundle $L$ on ${\Xg}$ that provides on each fibre $X$ a theta 
divisor and normalize $L$ such that it vanishes on the zero section of
${\Xg}$ over $\A{g}$. Then the Grothendieck-Riemann-Roch Theorem tells us
that
$$
\begin{aligned}
{\rm ch}(\pi_{!}L) &= \pi_{*}({\rm ch}(L) \cdot 
{\rm Td}^{\vee}(\Omega^1_{{\Xg}/\A{g}}))\cr
&=\pi_{*}({\rm ch}(L) \cdot 
{\rm Td}^{\vee}(\pi^*({\bE_g})))\cr
&= \pi_{*}({\rm ch}(L)) \cdot {\rm Td}^{\vee}({\bE}),\cr
\end{aligned} \eqno(1)
$$
where we used  $\Omega^1_{{\Xg}/\A{g}} \cong \pi^*({\bE})$ and
the projection formula. 
Here ${\rm Td}^{\vee}$ is defined for a line bundle with first Chern class 
$\gamma $ by $\gamma/(e^{\gamma}-1)$.
But $R^i\pi_*L=0$ for our relatively
ample $L$ for $i>0$ and so $\pi_{!}L$ is represented by a vector bundle,
and because $L$ defines a principal polarization, by a line bundle.
We write $\theta=c_1(\pi_{!}L)$. This gives the identity
$$
\sum_{k=0}^{\infty} \frac{\theta^k}{k!}= 
\pi_{*}(\sum_{k=0}^{\infty} \frac{\Theta^{g+k}}{(g+k)!}) \, 
{\rm Td}^{\vee}({\bE}). \eqno(2)
$$
Recall that 
${\rm Td}^{\vee}({\bE})=1-\lambda_1/2+ (\lambda_1^2+\lambda_2)/12+ \ldots $. 
We can compare terms of equal codimension; the codimension $1$ term gives
$$
\theta=-\lambda_1/2 +\pi_{*}(\Theta^{g+1}/(g+1)!). \eqno(3)
$$
If we now look how both sides of (2) 
behave when we replace $L$ by $L^n$, we
see immediately that the term $\Theta^{g+k}$ of the right hand side changes 
by a factor $n^{g+k}$. As to the left hand side, for a principally polarized 
abelian variety $X$ the space of sections $H^0(X,L_X^{\otimes n})$ is a representation 
of the Heisenberg group; it is the irreducible representation of degree $n^g$.
This implies that ${\rm ch}(\pi_{!}L)=n^g {\rm ch}(\pi_{!}L)$.
We see
$$
n^g \, \sum_{k=0}^{\infty} \frac{\theta^k}{k!} = \pi_*\left(\sum_{k=0}^{\infty} 
\frac{ n^{g+k} \Theta^{g+k}}{(g+k)!}\right) \cdot {\rm Td}^{\vee}({\bE}).
$$
Comparing the coefficients leads immediately to the following result.
\begin{corollary} In ${\CHQ}(\A{g})$ we have
$\pi_*({\rm ch}(L))=1$ and ${\rm ch}(\pi_*L)={\rm Td}^{\vee}({\bE})$.
\end{corollary}
In particular $\pi_*(\Theta^{g+1})=0$. Substituting this in (3)
gives the following result.

\begin{corollary} (Key Formula) We have $2\theta= -\lambda_1$.
\end{corollary}
\begin{corollary}
We have ${\rm ch}_{2k}({\bE})=0$ for $k\geq 1$.
\end{corollary}
\begin{proof}
The relation (1) reduces now to
$$
e^{-\lambda_1/2}= {\rm Td}^{\vee}({\bE}).
$$
If, say $\rho_1,\ldots,\rho_g$ are the Chern roots of ${\bE}$, so that
our $\lambda_i$ is the $i$th elementary symmetric function in the $\rho_j$,
then this relation is
$$
\prod_{i=1}^g \frac{e^{\rho_i}-e^{-\rho_i}}{\rho_i} = 1\, .
$$
This is equivalent with ${\rm ch}_{2k}({\bE})=0$ for $k>0$ and also with
${\rm Td}({\bE}\oplus {\bE}^{\vee})=0$.
\end{proof}

The proof of Theorem \ref{lambdagiszero} is also an application of 
Grothendieck-Riemann-Roch, this time applied to the structure sheaf
(i.e., $n=0$ in the preceding case). We find
$$
{\rm ch}(\pi_! {\mathcal O}_{\Xg})= 
\pi_{*}({\rm ch}({\mathcal O}_{\Xg})\cdot
{\rm Td}^{\vee}\pi^*({\bE}_g))= \pi_{*}(1) {\rm Td}^{\vee}({\bE}_g)=0.
$$
For an abelian variety $X$ the cohomology group $H^i(X,{\mathcal O}_X)$ 
is the $i$th 
exterior power of $H^1(X,{\mathcal O}_X)$ and  we thus
see that the left hand side equals
(global Serre duality)
$$
{\rm ch}(1-{\bE}^{\vee}+ \wedge^2{\bE}^{\vee}- \cdots 
+(-1)^g \wedge^g {\bE}^{\vee}).
$$
But by some general yoga (see \cite{BS}) we have for a vector bundle $B$ 
of rank $r$ the relation
$ \sum_{j=0}^r (-1)^j {\rm ch}(\wedge^j B^{\vee})= c_r(B) {\rm Td}(B)^{-1}$.
So the left hand side is $\lambda_g\,  {\rm Td}({\bE})^{-1}$ and since
${\rm Td}({\bE})$ starts with $1+\ldots$ it follows that $\lambda_g$ is zero. 

Theorem \ref{lambda1isample} is a classical result in characteric zero.
If we present our moduli space $\A{g}$ over ${\bC}$ as the quotient
space ${\SpgZ} \backslash {\Hg}$ then the determinant of the Hodge bundle
is a (n orbifold) line bundle that corresponds to the factor of automorphy
$$
\det (c\tau +d) \qquad \text{ for $\tau \in {\Hg}$ and $
\left(\begin{matrix} a & b \\ c & d \\ \end{matrix} \right) \in {\SpgZ}$.}
$$
In other words, a section of the $k$th power of $\det({\bE})$ gives by
pull back to ${\Hg}$ a holomorphic function $f: {\Hg} \to {\bC}$
with the property that
$$
f((a\tau+b)(c\tau+d)^{-1}) = \det (c\tau+d)^k f(\tau).
$$
A very well-known theorem by Baily and Borel (see \cite{B-B})
says that modular forms
of sufficiently high weight $k$ define an embedding of 
 ${\SpgZ} \backslash {\Hg}$ 
into projective space. The idea is that one can construct sufficiently many
modular forms to separate points and tangent vectors on $\A{g}({\bC})$. 
So $\lambda_1$ is an ample class. 
Clearly this holds then also in characteristic $p$
if $p$ is sufficiently large. So if we have a complete subvariety of
of codimension $j$ in $\A{g}\otimes {\bF}_p$ for every $p$ then
$\lambda_1^{j}$ cannot vanish when restricted to this subvariety.

Theorem \ref{lambda1isample} was extended to all characteristics by Moret-Bailly, see \cite{M-B}. 

\section{The Tautological Ring of ${\barA{g}}$}
In this section, just as in the preceding one, we will write $\A{g}$
for $\A{g}\otimes k$ with $k$ an algebraically closed field.

The class $\lambda_g$ vanishes in the rational Chow ring of $\A{g}$.
However, in a suitable compactification of $\A{g}$ this is no longer true
and instead of the quotient $R_{g-1}=R_g/(u_g)$ we find a copy of $R_g$
in the Chow ring of such a compactification.

We consider a smooth toroidal compactification $\barA{g}$
of $\A{g}$ of the type constructed by Faltings and Chai.  Over $\barA{g}$
we have a `universal' semi-abelian variety ${\cG}$ with a zero-section $s$,
see \cite{F-C}. Then $s^*{\rm Lie}({\cG})$ defines in a canonical way
an extension of the Hodge bundle ${\bE}$ on $\A{g}$ to a vector bundle
$\barA{g}$. We will denote this extension again by ${\bE}$. 

The relation (1) of the preceding section
can now be extended to ${\barA{g}}$.  A proof of this extension 
was given by Esnault and Viehweg, see \cite{E-V}. They show that in
characteristic $0$ for the Deligne extension of the cohomology sheaf
${\mathcal H}^1$ with its Gauss-Manin connection the Chern character
vanishes in degree $\neq 0$ by applying Grothendieck-Riemann-Roch to a
log extension and using the action of $-1$ to separate weights. 
By specializing one finds the following result.

\begin{theorem}\label{extendedlambdarelation}
The Chern classes $\lambda_i$ of the Hodge bundle ${\bE}$ on ${\barA{g}}$
satisfy the relation
$$
(1+\lambda_1+\cdots + \lambda_g)(1-\lambda_1+\cdots +(-1)^g \lambda_g)=1.
$$
\end{theorem}

This implies the following result:

\begin{theorem}
The tautological subring $\tilde{T}_g$ of ${\CHQ}({\barA{g}})$ 
is isomorphic to $R_g$
\end{theorem}
\begin{proof}
Clearly, by the relation of \ref{extendedlambdarelation} the tautological 
ring is a quotient of $R_g$; since $\lambda_1$ is ample on an open dense part
($\A{g}$) the socle 
$$
\lambda_g\lambda_{g-1}
\cdots \lambda_1 = \frac{1}{(g(g+1)/2)!} \left( \prod_{k=1}^g (2k-1)!! \right)
\lambda_1^{g(g+1)/2}
$$ 
does not vanish, 
and the tautological ring must be isomorphic to $R_g$.
\end{proof}

The Satake compactification $\SatAg$ of $\A{g}$, defined in general as the proj of the ring of Siegel modular forms, 
possesses a stratification
$$
\SatAg= \A{g} \sqcup {\cA}_{g-1} \sqcup \cdots \sqcup {\cA}_1 \sqcup {\cA}_0.
$$
By the natural map $q: \barA{g} \to \SatAg$ this induces a stratification of
${\barA{g}}$: we let ${\cA}_g^{(t)}$ be the inverse image $q^{-1}(
{\cA}_g \sqcup {\cA}_{g-1} \sqcup \cdots {\cA}_{g-t})$, the moduli space of
abelian varieties of torus rank $\leq t$. 
We have the following extension of Theorems \ref{lambdagiszero} and 
\ref{completecodimg}.

\begin{theorem}
For $t<g$ we have the relation
 $\lambda_g\lambda_{g-1} \cdots \lambda_{g-t}=0$
in the Chow group ${\rm CH}_{\bQ}^m({\cA}_g^{(t)})$ with $m=\sum_{i=0}^t g-i$.
\end{theorem}

\begin{theorem}
The moduli space ${\cA}_g^{(t)}\otimes {\bF}_p$ in characteristic $p>0$
contains a complete subvariety of codimension $g-t$, namely 
the locus of abelian
varieties of $p$-rank~$\leq t$.
\end{theorem}

\begin{corollary}
We have $\lambda_1^{g(g-1)/2+t} \neq 0 $ on ${\cA}_g^{(t)}$.
\end{corollary}

\begin{proposition}
For $0\leq k \leq g$ and $r>0$ we have on the boundary 
$\SatA{k} \subset \SatA{g}$  of the Satake compactification
$$
\lambda_1^{k(k+1)/2} \neq 0 \qquad \text{and} \qquad
\lambda_1^{k(k+1)/2 +r}=0 .
$$
\end{proposition}
This follows from the fact that $\lambda_1$ is an ample class on $\SatA{g}$,
that $\lambda_1 | \SatA{k}$ is again $\lambda_1$ (now of course defined
on $\SatA{k}$) and the fact that $\dim \SatA{k}= k(k+1)/2$. 

\section{The Proportionality Principle}

The result on the tautological ring has the following immediate corollary,
a special case of the so-called Proportionality Principle of 
Hirzebruch-Mumford:

\begin{theorem}
The characteristic numbers
of the Hodge bundle are proportional to those of the tautological bundle 
on the `compact dual' $Y_g$:
$$
\lambda_1^{n_1} \cdots \lambda_g^{n_g}([{\barA{g}}])=
(-1)^{g(g+1)/2} \frac{1}{2^g} 
\left( \prod_{k=1}^g \zeta(1-2k)
\right) \cdot u_1^{n_1}\cdots u_g^{n_g} ([Y_g])
$$
for all $(n_1,\ldots,n_g)$ with $\sum n_i= g(g+1)/2$.
\end{theorem}

Indeed, any top-dimensional class $\lambda_1^{n_1}\cdots \lambda_g^{n_g}$
is a multiple $m(n_1,\ldots,n_g)$ times the top-dimensional class
$\lambda_1 \cdots \lambda_g$ where the coefficient depends only on the structure
of $R_g$. So the proportionality principle is clear, and to make it explicit
one must find the value of one top-dimensional class, 
say $\lambda_1 \cdots \lambda_g$ on ${\barA{g}}$.
The top Chern class of the bundle ${\rm Sym}^2({\bE})$  on $\A{g}$ equals
$$
2^g \, \lambda_1 \cdots \lambda_g \, .
$$
We can interpret $2^g \deg (\lambda_1 \cdots \lambda_g)$   up to a sign
$(-1)^{g(g+1)/2}$ 
as the log version of the Euler number of $\A{g}$. 
It is known via the work of Siegel and Harder:
\begin{theorem} (Siegel-Harder) The Euler number of $\A{g}$ is equal to
$$
\zeta(-1)\zeta(-3) \cdots \zeta(1-2g).
$$
\end{theorem}

\begin{example}
For $g=1$ the Euler number of $\A{1}({\bC})$ is $-1/12$. It is obtained by 
integrating 
$$
(-1)\frac{1}{2\pi i} \frac{dx \wedge dy}{y^2} 
$$
over the usual fundamental domain ($\{ \tau=x+iy \in \mathfrak{H}: |x| \leq 1/2,
x^2+y^2 \geq 1\}$) for the action of ${\rm SL}(2,{\bZ})$
and (which gives $2\zeta(-1)=-1/6$) and
then dividing by $2$ since the group ${\rm SL}(2,{\bZ})$ does not
act effectively ($-1$ acts trivially).
\end{example}

The Proportionality Principle immediately extends to our equivariant bundles
$V_{\rho}$ since these are obtained by applying Schur functors to
powers of the Hodge bundle. 

\begin{corollary} (The Proportionality Principle of Hirzebruch-Mumford)
The characteristic numbers of the equivariant vector bundle $V_{\rho}$
on $\tilde{\A{g}}$ are proportional to those of the corresponding bundle
on the compact dual~$Y_g$
\end{corollary}

The proportionality factor here is
$$
p(g)= (-1)^{g(g+1)/2} \prod_{j=1}^g \frac{\zeta(1-2j)}{2}
$$
We give a few values.

\smallskip
\vbox{
\bigskip\centerline{\def\quad{\hskip 0.6em\relax}
\def\quod{\hskip 0.5em\relax }
\vbox{\offinterlineskip
\hrule
\halign{&\vrule#&\strut\quod\hfil#\quad\cr
height2pt&\omit&&\omit&&\omit&&\omit&&\omit&&\omit&\cr
&$g$&&$1$&&$2$&&$3$&&$4$&&$5$&\cr
height2pt&\omit&&\omit&&\omit&&\omit&&\omit&&\omit&\cr
\noalign{\hrule}
height2pt&\omit&&\omit&&\omit&&\omit&&\omit&&\omit&\cr
&${1/p(g) } $&&$24$&&$5760$&&$2903040$&&$1393459200$&&$367873228800$&\cr
height2pt&\omit&&\omit&&\omit&&\omit&&\omit&&\omit&\cr
} \hrule}
}}
\bigskip
A remark about the history of this `principle'.
Hirzebruch, inspired by a paper of Igusa, found in 1958 (see \cite{H:PT})
that for a discrete torsion-free group $\Gamma$
of automorphisms of a bounded symmetric domain $D$ such that the quotient
$X_{\Gamma}=\Gamma\backslash D$ is compact, the Chern numbers of the quotient 
are proportional to the Chern numbers of the compact dual $\hat{D}$ of $D$.
In the case of a subgroup $\Gamma \subset {\rm Sp}(2g,{\bR})$ that acts 
freely on ${\Hg}$ this means that there is a proportionality factor 
$e_{\Gamma}$ such that for all $n=(n_1,\ldots, n_{g(g+1)/2})$ 
with $n_i\in {\bZ}_{\geq 0}$ and $\sum n_i=g(g+1)/2$ we have
$$
c^n(X_{\Gamma})= 
e_{\Gamma}\,  c^n(Y_g) \, ,
$$
where $c^n$ stands for the top Chern class $\prod_i c_i^{n_i}$
and $Y_g$ is the compact dual of $\mathfrak{H}_g$.
His argument was that the top Chern classes of $X$ are represented by
$G({\bR})$-invariant differential forms of top degree and these are
proportional to each other on $D$ and on $\hat{D}$. The principle
extends to all equivariant holomorphic vector bundles on ${\Hg}$.
Mumford has obtained (see \cite{Mumford:HPP}) 
an extension of Hirzebruch's Proportionality 
Principle that applies to his toroidal compactifications of $\A{g}$. 

\section{The Chow Rings of $\barA{g}$ for $g=1,2,3$}\label{Chowringssmallg}
For low values of $g$ the Chow rings of $\A{g}$ and their toroidal
compactifications can be explicitly described.

\begin{theorem}
The Chow ring $\CHQ({\barA1})$ is isomorphic to 
${\bQ}[\lambda_1]/(\lambda_1^2)$.
\end{theorem}

Mumford determined in \cite{Mumford:Enumerative} the Chow ring of $\barA{2}$.
\begin{theorem}
The Chow ring ${\CHQ}(\barA2)$ is generated by 
$\lambda_1$, $\lambda_2$ and $\sigma_1$ and isomorphic to 
${\bQ}[\lambda_1,\lambda_2,\sigma_1]/(I_2),$ with $I_2$ the ideal generated by
$$
(1+\lambda_1+\lambda_2)(1-\lambda_1+\lambda_2)-1, \quad \lambda_2\sigma_1, 
\quad \sigma_1^2-22\lambda_1\sigma_1+120 \lambda_1^2.
$$
The ranks of the Chow groups are $1,2,2,1$.
\end{theorem}
The intersection numbers in this ring are given by

\medskip\centerline{\def\quad{\hskip 0.6em\relax}
\def\quod{\hskip 0.5em\relax }
\vbox{\offinterlineskip
\hrule
\halign{&\vrule#&\strut\quod\hfil#\quad\cr
height2pt&\omit&&\omit&&\omit&\cr
&$2\backslash 1$&&$\lambda_1$&&$\sigma_1$&\cr
height2pt&\omit&&\omit&&\omit&\cr
\noalign{\hrule}
height2pt&\omit&&\omit&&\omit&\cr
&$\lambda_1^2$&&$1/2880$&&$0$&\cr
&$\lambda_1\sigma_1$&&$0$&&$-1/24$&\cr
height2pt&\omit&&\omit&&\omit&\cr
\cr } \hrule}
}
\medskip
\par
\noindent
while the degrees of the top classes in the sigmas are
$$
\deg(\sigma_3)=1/4, \quad \deg(\sigma_2\sigma_1) =-1/4, \quad
\deg(\sigma_1^3)=-11/12.
$$
For genus $3$ the result is (cf.\ \cite{vdG:CHA3}):

\begin{theorem}
The tautological ring of ${\cA}_3$ is generated by the Chern classes 
$\lambda_i$ and isomorphic to ${\bQ}[\lambda_1]/(\lambda_1^4)$.
\end{theorem}

\begin{theorem}
The Chow ring ${\CHQ}(\barA3)$ of $\tilde{\mathcal A}_3$ is generated by the
$\lambda_i$, $i=1,2,3$ and the $\sigma_i$, $i=1,2$ and is
isomorphic to the graded ring (subscript is degree)
$$
{\bQ}[\lambda_1,\lambda_2,\lambda_3,\sigma_1,\sigma_2]/I_3,
$$
with $I_3$ the ideal generated by the relations
\begin{align}\notag
&(1+\lambda_1+\lambda_2+\lambda_3)(1-\lambda_1+\lambda_2-\lambda_3)=1, \cr
&\lambda_3\sigma_1=\lambda_3\sigma_2=\lambda_1^2\sigma_2=0,\cr
&\sigma_1^3= 2016 \, \lambda_3 - 4 \, \lambda_1^2 \sigma_1 -24
\, \lambda_1\sigma_2 +\frac{11}{3}\, \sigma_2\sigma_1,\cr
&\sigma_2^2= 360\,  \lambda_1^3\sigma_1 -45 \, \lambda_1^2\sigma_1^2
 +15 \, \lambda_1\sigma_2\sigma_1,\cr
&\sigma_1^2\sigma_2=1080\, \lambda_1^3 \sigma_1-165 \, \lambda_1^2\sigma_1^2+
47\, \lambda_1\sigma_2\sigma_1.\cr
\end{align}
The ranks of the Chow groups are: $1, 2, 4, 6, 4, 2, 1$.
\end{theorem}
For the degrees of the top dimensional elements we refer to \cite{vdG:CHA3}.

\section{Some Intersection Numbers}
As stated in section \ref{Ag} 
 the various compactifications employed for $\A{g}$
each have their own merits. For example, the toroidal compactification 
associated to the perfect cone decomposition has the advantage that its boundary
is an irreducible divisor $D=D_g$. 

By a result of Borel \cite{Borel} it is known that in degrees $\leq g-2$
the cohomology of ${\rm Sp}(2g,{\bQ})$ (and in fact for any finite index
subgroup $\Gamma$) is generated by elements in degree $2k+4$ for
$k=0,1,\ldots, (g-6)/2$:
$$
H^*(\Gamma, {\bQ})= {\bQ}[x_2,x_6,x_{10},\ldots] \, .
$$
In particular in degree $2$ there is one generator. We deduce (at least for
$g\geq 6$ and with special arguments also for $2 \leq g \leq 5$):
\begin{proposition}
The Picard group of $\A{g}$ for $g\geq 2$ 
is generated by the determinant $\lambda_1$
of the Hodge bundle.
\end{proposition} 
By observing that on $\A{g}^{(1)}$ and $\A{g}^{\rm Perf}$ the boundary divisor
is irreducible we get:

\begin{corollary} The Picard group of $\A{g}^{(1)}$ and of 
${\cA}_g^{\rm Perf}$ is generated by $\lambda_1$ and the class of $D$.
\end{corollary}

It would be nice to know the top intersection numbers $\lambda_1^n D^{G-n}$
with $G=g(g+1)/2=\dim \A{g}$. It seems that these numbers are zero
when $n$ is not of the form $k(k+1)/2$. In fact, Erdenberger, Grushevsky 
and Hulek formulated
this as a conjecture, see \cite{E-G-H}.

\begin{conjecture}
The intersection number $\deg \lambda_1^n\, D^{G-n}$ on ${\cA}_g^{\rm Perf}$
vanishes unless $n$ is of the form $n=k(k+1)/2$ for $k\leq g$.
\end{conjecture}
From our results above it follows that
$$
\deg \lambda_1^G = 
(-1)^G \frac{G!}{2^g} \prod_{k=1}^g \frac{\zeta(1-2k)}{(2k-1)!!}\, .
$$
Erdenberger, Grushevsky and Hulek calculated the next two cases. We
quote only the first
$$
\deg \lambda_1^{g(g-1)/2} D^g = (-1)^{G-1}\frac{(g-1)!(G-g)!}{2} 
\prod_{k=1}^{g-1} \frac{\zeta(1-2k)}{(2k-1)!!}
$$
and we refer for the intersection number
$\deg \lambda^{(g-2)(g-1)/2} D^{2g-1}$ to loc.\ cit.
\section{The Top Class of the Hodge Bundle}
As before, we shall write here $\A{g}$ for $\A{g} \otimes k$ 
with $k$ an algebraically closed field.

The cycle class $\lambda_g$ vanishes in $\CHQ^g(\A{g})$. 
However, it does not vanish on $\CHQ^g(\barA{g})$, 
for example because $\lambda_1\lambda_2\cdots \lambda_g$ is a 
non-zero multiple of $\lambda_1^{g(g+1)/2}$ that has positive degree
 since $\lambda_1$ is ample on $\A{g}$.
This raises two questions. First,
what is the order of $\lambda_g$ in ${\CH}^g(\A{g})$ as a torsion class?
Second, since up to torsion $\lambda_g$ comes from the boundary $\barA{g}-
\A{g}$, one can ask for a natural supporting cycle for this class in the
boundary. Since we work on stacks one has to use intersection theory
on stacks; the theory is still in its infancy, but we use Kresch's
approach \cite{Kresch}. Mumford answered the first question for $g=1$
in \cite{Mumford:12L} .

In joint work with Ekedahl (\cite{E-vdG:TOP,E-vdG:lambdag})
we considered these two questions.
Let us begin with the torsion order of $\lambda_g$ in ${\rm CH}^g(\A{g})$.
A well-known formula of Borel and Serre 
(used above in Section \ref{The Hodge Bundle}) says that
the Chern character of the alternating sum of the exterior powers
of the Hodge bundle satisfies
$$
{\rm ch}(\wedge^*{\bE})= (-1)^g \lambda_g {\rm Td}({\bE})^{-1}
$$
and this  implies that its terms of degree $1, \ldots, g-1$ vanish and
that it equals $-(g-1)! \lambda_g$ in degree $g$.

\begin{lemma}
Let $p$ be a prime and $\pi: A \to S$ be a family of abelian varieties
of relative dimension $g$ and $L$ a line bundle on $A$ that is of order $p$
on all fibres of $\pi$. If $p> \min(2g,\dim S +g)$ then $p \, (g-1)! \lambda_g =0$.
\end{lemma}
Indeed, we may assume after twisting by a pull back from the base $S$
that $L^p$ is trivial. Let $[L]$ be its class in $K_0(A)$. Then
$$
0=[L]^p-1= ([L]-1)^p+p\, ([L]-1)(1+\frac{p-1}{2}([L]-1)+\ldots )\, .
$$
Now $[L]-1$ is nilpotent because it has support of codimension $\geq 1$ and
so it follows that $p\, ([L]-1)$ lies in the ideal generated by 
$([L]-1)^p$ which has codimension $\geq p$. Now if $p>2g$ or $p> \dim S +g$ the
codimension is $> 2g$, hence the image under $\pi$ has codimension $>g$
or is zero. We thus can safely remove it and may assume that $p\, [L]=p$
in $K_0(A)$. Consider now the Poincar\'e bundle $P$ on $A \times \hat{A}$;
we know that $H^i(X \times \hat{X},P)$ for an abelian variety $X$ is 
zero for $i\neq g$ and $1$-dimensional for $i=g$. So the derived pullback
of $R\pi_*P$ along the zero section of $A\times \hat{A}$ over $A$ is $R\pi_*{\mathcal O}_A$. We know that $p\, [P]=p[L\otimes P]$ and $p\, [R\pi_*P]=
p\, [R\pi_*(L\otimes P)]$ and $R\pi_*(L\otimes P)$ has support 
along the inverse of the section of $\hat{A}$ corresponding to $L$. 
This section is
everywhere disjoint from the zero section, so the pull back of
$R\pi_*(L\otimes P)$ along the zero section is trivial:
$p\, [R\pi_* {\mathcal O}_a]=0$ and since $R\pi_*{\mathcal O}_A=\wedge^*
R^1\pi_*{\mathcal O}_A=\wedge^* {\bE}$ we find $-p (g-1)! \lambda_g=0$. 

We put
$$
n_g:= {\rm gcd}\{ p^{2g}-1: \text{ $p$ running through the primes $>2g+1$} \} \, .
$$
Note that for a prime $\ell$ we have $\ell^k | n_g$ if and only if
the exponent of $({\bZ}/\ell^k{\bZ})^*$ divides $2g$. 
By Dirichlet's prime number theorem we have for $p>2$
$$
{\rm ord}_p(n_g)= \begin{cases}
0 & (p-1)\not| 2g \\
\max \{ k : p^{k-1}(p-1) | 2g \} & (p-1)|2g, \\
\end{cases}
$$
while
$$
{\rm ord}_2(n_g)= \max \{ k: 2^{k-2} | 2g\}.
$$
For example, we have 
$$
n_1=24, \quad n_2=240, \quad n_3=504, \quad n_4=480.
$$
\begin{theorem} 
Let $\pi: \X{g}\to \A{g}$ be the universal family. 
Then $(g-1)! n_g \lambda_g =0$.
\end{theorem}
For the proof we consider the commutative diagram
$$
\begin{matrix}
\X{g}' & \longrightarrow & \X{g} \\
\downarrow && \downarrow \\
\A{g}' & \longrightarrow & \A{g} \\
\end{matrix}
$$
where $\A{g}'\to \A{g}$ is the degree $p^{2g}-1$ cover 
obtained by adding a line bundle
of order $p$. It follows that $(g-1)! p (p^{2g}-1) \lambda_g$ vanishes.
Then the rest follows from the definition of $n_g$.

In \cite{Mumford:12L} Mumford proved that the order of $\lambda_1$ is $12$.
So our result is off by a factor $2$.
\smallskip

In the paper \cite{E-vdG:TOP} we also gave the vanishing orders for
the Chern classes of the de Rham bundle $\HDR$ for the universal family
$\X{g}\to \A{g}$. This bundle is provided with an integrable 
connection and so its Chern classes are torsion classes in integral
cohomology. If $l$ is a prime different from the characteristic
these Chern classes are torsion too by the comparison theorems and
using specialization. 
We denote these classes by $r_i \in H^{2i}(\A{g}, {\bZ}_l)$. We determined 
their orders up to a factor $2$. Note that $r_i$ vanishes for $i$ odd
as $\HDR$ is a symplectic bundle. 

\begin{theorem}
For all $i$ the class $r_{2i+1}$ vanishes. For $1\leq i \leq g$ 
the order of $r_{2i}$ in 
integral  (resp.\ $l$-adic) cohomology  equals (resp.\ equals the
$l$-part of) $n_i/2$ or $n_i$.
 \end{theorem}

\bigskip
We now turn to the question of a representative cycle in the
boundary for $\lambda_g$ in $\CHQ(\barA{g})$.

It is a very well-known fact that there is a cusp form 
$\Delta$ of weight $12$ on ${\rm SL}(2,{\bZ})$ and that it 
has a only one zero, a simple zero at the cusp. This leads to
the relation $12 \, \lambda_1= \delta$, where $\delta$ is the cycle 
representing the class of the cusp. 
This formula has an analogue for higher $g$.

\begin{theorem}
In the Chow group ${\CHQ}({\cA}^{(1)}_g)$ 
of codimension $g$ cycles
on the moduli stack $\A{g}^{(1)}$ 
of rank $\leq 1$ degenerations 
the top Chern class $\lambda_g$ satisfies the formula
$$
\lambda_g= (-1)^g \zeta(1-2g) \, \delta_g,
$$
with $\delta_g$ the ${\bQ}$-class of the locus $\Delta_{g}$ of 
semi-abelian varieties which are trivial extensions of an abelian 
variety of dimension $g-1$ by the multiplicative group~${\bG}_m$.
\end{theorem}

The number $\zeta(1-2g)$ is a rational number $-b_{2g}/2g$ with $b_{2g}$
the $2g$th Bernoulli number. For example, we have 
$$
12 \, \lambda_1= \delta_1, \quad 120\, \lambda_2= \delta_2, \quad 
252\, \lambda_3=\delta_3, \quad 240\, \lambda_4=\delta_4,
\quad 132 \, \lambda_5=\delta_5 
$$

One might also wish to work on the Satake compactification $\SatA{g}$,
singular though as it is. Every toroidal compactification of 
Faltings-Chai type $\barA{g}$ has a canonical morphism 
$q: \barA{g} \to \SatA{g}$. Then we can define a class $\ell_{\alpha}$
where $\alpha$ is a subset of $\{1,2,\ldots,g\}$ by
$$
\ell_{\alpha}= q_*(\lambda_{\alpha}) \quad
\text{with $\lambda_{\alpha}$ equal to 
$\prod_{i \in {\alpha}}\lambda_i\in {\rm CH}^{\bQ}_{d}(\barA{g})$}
$$
with $d=d(\alpha)=\sum_{i \in \alpha} i$ the degree of $\alpha$
and ${\rm CH}^{\bQ}_{d}(\SatA{g})$ the Chow homology group.
Note that this push forward does not depend on the choice of toroidal
compactification as such toroidal compactifications always allow a 
common refinement and the $\lambda_i$ are compatible with pull back.
One can ask for example whether the classes of the boundary components
$\SatA{j}$ lie in the ${\bQ}$-vector space generated by the $\ell_{\alpha}$
with $d(\alpha)={\rm codim}_{\SatA{g}}(\SatA{j})$.
In \cite{E-vdG:lambdag} we made the following conjecture.

\begin{conjecture}\label{classAg*conj}
In the group ${\rm CH}^{\bQ}_d(\SatA{g})$ with $d=g(g+1)/2-k(k+1)/2$
we have 
$$
[\SatA{g-k}]=
\frac{(-1)^k}{\prod_{i=1}^{k} \zeta(2k-1-2g)} \ell_{g,g-1,\ldots,g+1-k}
$$
\end{conjecture}

The evidence that Ekedahl and I gave is:
\begin{theorem}
Conjecture \ref{classAg*conj} is true for $k=1$ and $k=2$ and if 
${\rm char}(k)=p>0$ then for all $k$.
\end{theorem}

We shall see later that a multiple of the class $\lambda_g$ has a beautiful
representative cycle in $\A{g}\otimes {\bF}_p$, namely the locus of
abelian varieties of $p$-rank $0$. 

\section{Cohomology}

By a result of Borel \cite{Borel} the stable cohomology of the symplectic group is known; this implies that in degrees $\leq g-2$
the cohomology of ${\rm Sp}(2g,{\bQ})$ (and in fact for any finite index
subgroup $\Gamma$) is generated by elements in degree $2k+4$ for
$k=0,1,\ldots, (g-6)/2$: 
$$
H^*(\Gamma, {\bQ})= {\bQ}[x_2,x_6,x_{10},\ldots]\, .
$$
The Chern classes $\lambda_{2k+1}$ of the Hodge bundle provide these classes.

There are also some results on the stable homology of 
the Satake compactification, see \cite{CL}; besides the $\lambda_{2k+1}$
there are other classes $\alpha_{2k+1}$ for $k\geq 1$.

Van der Kallen and Looijenga proved in \cite{vdK-L} that the rational homology
of the pair $(\A{g}, {\A{g,{\rm dec}}})$ with ${\A{g,{\rm dec}}}$ 
the locus of decomposable principally polarized abelian 
varieties, vanishes in degree $\leq g-2$.

For low values of $g$ the cohomology of $\A{g}$  is known.
For $g=1$ one has $H^0(\A{1},{\bQ})={\bQ}$ and $H^1=H^2=(0)$.
For $g=2$ one can show that $H^0(\A{2},{\bQ})={\bQ}$, 
$H^2(\A{2},{\bQ})={\bQ}(-1)$. 

Hain determined the cohomology of $\A{3}$. His result (\cite{Hain}) is:

\begin{theorem}
The cohomology $H^*(\A{3}, {\bQ})$ is given by
$$
H^{j}(\A{3}, {\bQ})\cong \begin{cases}
{\bQ} & j=0 \\
{\bQ}(-1) & j=2 \\
{\bQ}(-2) & j=4 \\
E & j=6 \\
0 & \text{else}
\end{cases}
$$
where $E$ is an extension 
$$
0\to {\bQ}(-3) \to E \to
{\bQ}(-6) \to 0
$$
\end{theorem}

We find for the compactly supported cohomology a similar result
$$
H_c^{j}(\A{3}, {\bQ})\cong \begin{cases}
{\bQ}(-6) & j=12\\
{\bQ}(-5) & j=10 \\
{\bQ}(-4) & j=8 \\
E' & j=6 \\
0 & \text{else}
\end{cases}
$$
where $E'$ is an extension $0\to {\bQ} \to E' \to {\bQ}(-3) \to 0$.

The natural map  $H_c^* \to H^*$ is the zero map. Indeed, the classes
in $H_c^j$ for $j=12$, $10$ and $8$ are $\lambda_3\lambda_1^3$, 
$\lambda_3 \lambda_1^2$ and $\lambda_3\lambda_1$, 
while $\lambda_3$ gives a non-zero class in $H^6_c$. On the other hand, $1$, $\lambda_1$, $\lambda_1$ and $\lambda_1^3$
give non-zero classes in $H^0$, $H^2$, $H^4$ and $H^6$. Looking at the weight 
shows that the map is the zero map. 

By calculating the cohomology of $\barA{3}$ 
Hulek and Tommasi proved in \cite{H-T} 
that the cohomology of the Voronoi compactification
for $g\leq 3$ coincides with the Chow ring (known by the Theorems
of Section \ref{Chowringssmallg}):

\begin{theorem}
The cycle class map gives an isomorphism
$$
\CHQ^{*}(\A{g}^{\rm Vor}) \cong H^{*}(\A{g}^{\rm Vor}) \quad 
\text{for $g=1,2,3$}\, .
$$
\end{theorem}

\section{Siegel Modular Forms}
The cohomology of $\A{g}$ itself and of the universal family $\X{g}$
and its powers $\X{g}^n$ is closely linked to modular forms.
We therefore pause to give a short description of these modular
forms on ${\rm Sp}(2g,{\bZ})$. For a general reference we refer
to the book of Freitag \cite{Freitag:SMF} and to \cite{vdG:SMF} and the
references there.

Siegel modular forms generalize the notion of usual (elliptic) modular
forms on ${\rm SL}(2, {\bZ})$ and its (finite index) subgroups.
We first need to generalize the notion of the weight of a modular
form. To define it we need a finite-dimensional complex representation 
$\rho : {\rm Gl}(g,{\bC}) \to {\rm GL}(V)$ with $V$ a finite-dimensional
complex vector space. 

\begin{definition}
A holomorphic map $f: \mathfrak{H}_g \to V$ is called a Siegel modular form
of weight $\rho$ if
$$
f(\gamma(\tau))= \rho(c\tau + d) f(\tau) \quad
\text{for all $\gamma = \left( \begin{matrix} a & b \\ c & d \\ \end{matrix} \right)
\in {\Sp}(2g,{\bZ})$ and all $\tau \in \mathfrak{H}_g$},
$$ 
plus for $g=1$ the
extra requirement that $f$ is holomorphic at the cusp. (The latter means that
$f$ has a Fourier expansion $f=\sum_{n\geq 0} a(n)q^n$ with 
$q=e^{2\pi i \tau}$.)
\end{definition}
Modular forms of weight $\rho$ form a ${\bC}$-vector space 
$M_{\rho}({\Sp}(2g,{\bZ}))$; as it turns out this space is
finite-dimensional.
This vector space can be identified with the space of holomorphic sections of 
the vector bundle $V_{\rho}$ defined before 
(see Section \ref{The Hodge Bundle}). Since we are working on an 
orbifold, one has to be careful; we could replace ${\rm Sp}(2g, {\bZ})$
by a normal congruence subgroup $\Gamma$ of finite index that acts freely 
on $\mathfrak{H}_g$, take the space $M_{\rho}(\Gamma)$ of 
modular forms on $\Gamma$  and consider the invariant subspace 
under the action of ${\rm Sp}(2g,{\bZ})/\Gamma$.

Since we can decompose the representation $\rho$ into irreducible representations it is no restriction of generality to limit ourselves to the case where
$\rho$ is irreducible. 
The irreducible representations $\rho$ of ${\rm GL}(g,{\bC})$
are given by $g$-tuples integers $(a_1,a_2,\ldots,a_g)$ with
$a_i \geq a_{i+1}$, the highest weight of the representation.

A special case is where $\rho$ is given by $a_i=k$, in other words
$$
\rho(c\tau+d)= \det(c\tau+d)^k \, .
$$ 
In this case the Siegel modular forms are called {\sl classical 
Siegel modular forms} of weight $k$. 

A modular form $f$ has a Fourier expansion
$$
f(\tau)= \sum_{n \, \text{ half-integral}} a(n) e^{2\pi i {\rm Tr}(n\tau)},
$$
where $n$ runs over the symmetric $g\times g$ matrices that are half-integral
(i.e. $2n$ is an integral matrix with even entries along the diagonal)
and 
$$
{\rm Tr}(n \tau)= \sum_{i=1}^g n_{ii} \tau_{ii} +2 \sum_{1 \leq i < j \leq g}
n_{ij} \tau_{ij}.
$$
A classical result of Koecher (cf.\ \cite{Freitag:SMF}) 
asserts that $a(n)=0$ for $n$ that are not
semi-positive. This is a sort of Hartogs extension theorem.

We shall use the suggestive notation $q^n$ for $e^{2 \pi i {\rm Tr}(n\tau)}$
and then can write the Fourier expansion as
$$
f(\tau)=\sum_{n\geq 0,\, \text{half-integral}} a(n) \, q^n.
$$
We observe the invariance property
$$
a(u^t n u)= \rho(u^t) a(n) \quad \text{for all $u \in {\rm GL}(g,{\bZ})$}
\eqno(1)
$$
This follows by the short calculation
\begin{align}
a(u^tnu)&=\int_{x \, \bmod 1} f(\tau)e^{-2\pi i {\rm Tr}(u^tnu \tau)} dx\cr
&=\rho(u^t) \int_{x \, \bmod 1}
 f(u\tau u^t)e^{-2 \pi i {\rm Tr}(n \, u\tau u^t)} dx= \rho(u^t) a(n),\cr \notag \end{align}
where we wrote $\tau=x+iy$.

There is a way to extend Siegel modular forms on $\A{g}$ to modular forms on
$\A{g}\sqcup \A{g-1}$ and inductively to $\SatA{g}$ by means of the so-called
$\Phi$-operator of Siegel. For $f \in M_{\rho}$ one defines
$$
\Phi(f)(\tau^{\prime})= \lim_{t \to \infty} f(\begin{matrix} 
\tau^{\prime} & 0 \\ 0 & it\\ \end{matrix}) \qquad \tau^{\prime} \in \mathfrak{H}_{g-1} \, .
$$
The limit is well-defined and gives a function that satisfies
$$
(\Phi f)(\gamma^{\prime}(\tau^{\prime}))
= \rho(\begin{matrix} c\tau +d & 0 \\ 0 & 1 \\
\end{matrix}) (\Phi f) (\tau^{\prime})\, ,
$$
where $\gamma^{\prime}=(a \, b ; c \, d) \in {\rm Sp}(2g-2, {\bZ})$ is embedded 
in ${\rm Sp}(2g,{\bZ})$ as the automorphism group of the symplectic subspace
$\langle e_i, f_i : i=1,\ldots,g-1\rangle$.
That is, we get a linear map
$$
M_{\rho} \longrightarrow M_{\rho^{\prime}}, \quad f \mapsto \Phi(f),
$$
for some representation $\rho'$; in fact,
with the representation $\rho^{\prime}=(a_1,a_2,\ldots,a_{g-1})$ for 
irreducible $\rho=(a_1,a_2,\ldots,a_{g})$, cf.\ the proof of 
\ref{corankresult}.

\begin{definition}
A Siegel modular form $f \in M_{\rho}$ is called a {\sl cusp form} if $\Phi(f)=0$.
\end{definition}

We can extend this definition by

\begin{definition}
A non-zero  modular form $f \in M_{\rho}$ has co-rank $k$ if
$\Phi^{k}(f)\neq 0$ and $\Phi^{k+1}(f) =0$.
\end{definition}
That is, $f$ has co-rank $k$ if it survives (non-zero) 
till  $\A{g-k}$ and no further.
So cusp forms have co-rank $0$. 

For an irreducible representation $\rho$
of ${\rm GL}_g$ with highest weight $(a_1,\ldots,a_g)$ 
we define the co-rank as 
$$
{\rm co-rank}(\rho)= \# \{ i: 1\leq i \leq g, a_i=a_g\} 
$$
Weissauer proved in \cite{Weissauer} the following result.

\begin{theorem}\label{corankresult} 
Let $\rho$ be irreducible. 
For a non-zero Siegel modular form $f \in M_{\rho}$ 
one has ${\rm corank}(f) \leq {\rm corank}(\rho)$.
\end{theorem}
\begin{proof}
We give Weissauer's proof. 
Suppose that $f: \mathfrak{H}_g \to V$ is a form of co-rank $k$ with Fourier series
$\sum a(n)q^n$; then there exists a
semi-definite half-integral matrix $n$ such that 
$$
n =\left( \begin{matrix} n^{\prime} & 0 \\ 0 & 0 \\ \end{matrix} \right)
\quad \text{with $n^{\prime}$ a $(g-k)\times (g-k)$ matrix} 
$$
such that $a(n)\neq 0$.
The identity (1) implies that $a(n)=\rho(u)a(n)$ for all $u$ in the group
$$
G_{g,k}= \{ \left(\begin{matrix} 1_{g-k} & b \\ 0 & d \\ \end{matrix}
\right): d \in {\rm GL}_k  \} 
$$
This group is a semi-direct product ${\rm GL}_k \ltimes N$ with $N$ the
unipotent radical. The important remark now is that the Zariski closure of
${\rm GL}(k,{\bZ})$ in $G_{g,k}$ contains ${\rm SL}_k({\bC}) \ltimes N$
and we have $a(n)=\rho(u)a(n)$ for all $u \in {\rm SL}_k({\bC}) \ltimes N$.
So the Fourier coefficient $a(n)$ lies in the subspace $V^{N}$
of $N$-invariants. The parabolic
group $P=\{ (a\, b ; 0 \, d) ; a \in {\rm GL}_{n-k}, d \in {\rm GL}_k\}$,
which is also a semi-direct product $({\rm GL}_{n-k}\times {\rm GL}_k) 
\ltimes N$, acts on $V^N$ via its quotient ${\rm GL}_{n-k}\times {\rm GL}_k$.
Since this is a reductive group $V^N$ decomposes into irreducible
representations, each of which is a tensor product of an irreducible
representation of  ${\rm GL}_{n-k}$ times an irreducible
representation of  ${\rm GL}_{k}$.
If $U_k$ denotes the subgroup of upper triangular unipotent matrices in
${\rm GL}(k,{\bC})$, 
then the highest weight space is $(V^N)^{U_{g-k} \times U_k}$
and this equals $V^{U_g}$, and this is $1$-dimensional. Thus $V^N$
is an irreducible representation of ${\rm GL}_{g-k} \times {\rm GL}_k$
and one checks that it is given by $(a_1,\ldots, a_{g-k}) \otimes (a_{g-k+1},
\ldots, a_g)$. The space of invariants $V^{G_{g,k}} \subset V^N$
under $G_{g,k}$ can only contain ${\rm SL}_k({\bC})$-invariant elements 
if the ${\rm GL}_k$-representation is $1$-dimensional. 
Therefore $V^{G_{g,k}}$ is zero unless $ (a_{g-k+1},\ldots, a_g)$ is a
$1$-dimensional representation, hence $a_{g-k+1}=\ldots = a_g$.
In that case the representation for ${\rm GL}_{g-k}$ is given by 
$(a_1,\ldots, a_{g-k})$. Hence the Fourier coefficients of $f$ all have
to vanish if  the corank of $f$ is greater than the corank of $\rho$.
 This proves the result.
\end{proof}
\section{Differential Forms on $\A{g}$}\label{DF}
Here we look first at the moduli space over the field of complex numbers
$\A{g}({\bC})={\Sp}(2g,{\bZ}) \backslash {\Hg}$. We are interested  
in differential forms living on $\barA{g}$. If $\eta$ is a holomorphic
$p$ form on $\barA{g}$ then we can pull the form back to $\mathfrak{H}_g$.
It will there be a section of some exterior power of $\Omega^1_{\mathfrak{H}_g}$,
hence of some exterior power of the second symmetric power of the
Hodge bundle. Such forms can be analytically described by vector-valued
Siegel modular forms.

As we saw in \ref{OmegaisSym2}, 
the bundle $\Omega^1(\log D)$ is associated to the 
second symmetric power ${\rm Sym}^2({\bE})$ of the Hodge bundle 
(at least in the stacky interpretation) and we are led to ask 
which irreducible representations occur in the exterior powers 
of the second symmetric power of the
standard representation of ${\rm GL}(g)$. The answer is that these
are exactly those irreducible representations $\rho$ that are of the form
$ w \eta -\eta$, where $\eta=(g,g-1,\ldots,1)$ is half the sum of the 
positive roots and $w$ runs through the $2^g$ Kostant representatives
of $W({\rm GL}_g) \backslash W({\rm Sp}_{2g})$, where $W$ is the Weyl group.
They have the property that they send dominant weights for ${\rm Sp}_{2g}$
to dominant weights for ${\rm GL}_g$.

To describe this explicitly, recall that the Weyl group $W$ 
of ${\rm Sp}(2g,{\bZ})$
is the semi-direct product $S_g \ltimes ({\bZ}/2{\bZ})^g$ of signed
permutations. The Weyl group of ${\rm GL}_g$ is equal to the symmetric group
$S_g$. Every left coset $S_g\backslash W$ contains exactly one 
Kostant element $w$. Such an element $w$ corresponds one to one 
to an element of $({\bZ}/2{\bZ})^g$ that we view a $g$-tuple 
$(\epsilon_1,\ldots,\epsilon_g)$
with $\epsilon_i \in \{ \pm 1 \}$ 
and the action of this element $w$ on the root lattice is then via
$$
(a_1,\ldots,a_g) {\buildrel w \over \longrightarrow} 
(\epsilon_{\sigma(1)} a_{\sigma(1)}, \ldots, \epsilon_{\sigma(g)} a_{\sigma(g)})$$
for all $a_1 \geq a_2 \geq \cdots \geq a_g$ and with $\sigma$ the unique
permutation such that 
$$
\epsilon_{\sigma(1)} a_{\sigma(1)} \geq \cdots \geq 
\epsilon_{\sigma(g)} a_{\sigma(g)}.
$$

If $f$ is a classical Siegel modular form of weight $k=g+1$ on the group
$\Gamma_g$  then $f(\tau) \prod_{i\leq j} d\tau_{ij}$
is a top differential on the smooth part of quotient space
$\Gamma_g \backslash {\mathcal H}_g={\mathcal A}_g(\bC)$.
It can be extended over the smooth part of the
rank-$1$ compactification ${\mathcal A}_g^{(1)}$
if and only if $f$ is a cusp form. It is not difficult to see
that this form can be extended as a holomorphic form to the whole
smooth compactification $\tilde{\mathcal A}_g$.

\begin{proposition}\label{top}
The map that associates to a classical Siegel modular cusp form $f
\in S_{g+1}(\Gamma_g)$ of weight $g+1$ the top differential
$\omega= f(\tau) \prod_{i\leq j} d\tau_{ij}$ defines an
isomorphism between $S_{g+1}(\Gamma_g)$ and the space
of holomorphic top differentials
$H^0(\tilde{\mathcal A}_g, \Omega^{g(g+1)/2})$
on any smooth compactification
$\tilde{\mathcal A}_g$.
\end{proposition}

Freitag and Pommerening proved in \cite{FP} the following extension result.

\begin{theorem}
Let $p<g(g+1)/2$ and $\omega$ a holomorphic $p$-form on the smooth locus
of ${\rm Sp}(2g,{\bZ})\backslash \mathfrak{H}_{g}$. 
Then it extends uniquely to a holomorphic $p$-form on any smooth 
toroidal compactification $\barA{g}$.
\end{theorem}
But for $g>1$ the singular locus has codimension $\geq 2$. This implies:

\begin{corollary}
For $p<g(g+1)/2$ holomorphic $p$-forms on $\barA{g}$ correspond one-one 
to ${\rm Sp}(2g,{\bZ})$-invariant
holomorphic $p$-forms on $\mathfrak{H}_g$:
$$
\Gamma(\barA{g},\Omega^p)\cong (\Omega^p(\mathfrak{H}_g))^{{\rm Sp}(2g,{\bZ})}.
$$
\end{corollary}

When can holomorphic $p$-forms exist on $\barA{g}$?
Weissauer proved in \cite{Weissauer} a vanishing theorem of the following
form.

\begin{theorem}
Let $\tilde{\mathcal A}_g$ be a smooth compactification of
${\mathcal A}_g$. If $p$ is an integer $0\leq p< g(g+1)/2$ then
the space of holomorphic $p$-forms on
$\tilde{\mathcal A}_g$ vanishes unless $p$ is of the form
$g(g+1)/2-r (r+1)/2$ with $1\leq r \leq g$
and then
$H^0(\tilde{\mathcal A}_g,\Omega^p_{\tilde{\mathcal A}_g})
\cong M_{\rho}(\Gamma_g)$
with $\rho=(g+1,\ldots,g+1,g-r,\ldots,g-r)$ with
$g-r$ occuring $r$ times.
\end{theorem}

Weissauer deduced this from the following Vanishing Theorem for
Siegel modular forms (cf.\ loc.\ cit.).

\begin{theorem}
Let $\rho=(a_1,\ldots,a_g)$ be irreducible with ${\rm co-rank}(\rho) < g-a_g$.
If we have
$$
\# \{ 1\leq i \leq g : a_i=a_g+1 \} < 2(g-a_g-\text{\rm co-rank}(\rho))
$$
then $M_{\rho}=(0)$.
\end{theorem}
\section{The Kodaira Dimension of $\A{g}$}
The Kodaira dimension of the moduli space $\A{g}$ over ${\bC}$, 
the least integer
$\kappa=\kappa({\A{g}})$ such that the sequence $P_m/m^{\kappa}$ 
with $P_m({\A{g}})=
\dim H^0({\A{g}},K^{\otimes m})$, is bounded, is an important invariant.
In terms of Siegel modular forms this comes down to the growth of
the dimension of the space of classical Siegel modular forms $f$
of weight $k(g+1)$ that extend to holomorphic tensors 
$f(\tau) \eta^{\otimes k}$ with 
$\eta= \wedge_{1\leq i \leq j \leq g}  d \tau_{ij}$ on $\barA{g}$.
The first condition is that $f$ vanishes with multiplicity $\geq k$ along
the divisor at infinity. Then one has to deal with the extension over
the quotient singularities. Reid and Tai independently found a criterion
for the extension of pluri-canonical forms over quotient singularities.

\begin{proposition}\label{reid-tai} 
(Reid-Tai Criterion) A pluri-canonical form $\eta$
on ${\bC}^n$ that is invariant under a finite group $G$ acting linearly
on ${\bC}^n$ extends to a resolution of singularities if 
for every non-trivial element $\gamma \in G$ and every fixed point $x$
of $\gamma$ we have
$$
\sum_{j=1}^n \alpha_j \geq 1\, ,
$$
where the action of $\gamma$ on the tangent space of $x$ has eigenvalues
$e^{2 \pi i \alpha_j}$. If $G$ does not possess pseudo-reflections
then it extends  if and only if $\sum_{j=1}^n \alpha_j \geq 1$.
\end{proposition}

Tai checked (cf.\ \cite{Tai}) 
that if $g\geq 5$ then every fixed point of a non-trivially
acting element of ${\rm Sp}(2g,{\bZ})$ satisfies the requirement 
of the Reid-Tai criterion and thus these forms extend over the
quotient singularities of ${\A{g}}$. He also checked that the
extension over the singularities in the boundary did not present
difficulties. Thus he showed by calculating the dimension of the
space of sections of $K^{\otimes k}$ on a level cover of $\A{g}$
and calculating the space over invariants under the action of ${\rm Sp}(2g, {\bZ}/\ell {\bZ})$ that for $g \geq 9$ the space $\A{g}$ was of general type.
He thus improved earlier work of Freitag (\cite{Freitag:KD}).

Mumford extended this result in \cite{Mumford:KD} and proved:
\begin{theorem}\label{KDgatleast7}
The moduli space $\A{g}$ is of general type if $g \geq 7$.
\end{theorem}

His approach was similar to the method used in his joint paper with
Harris on the Kodaira dimension of $\M{g}$. First he works
on the moduli space $\A{g}^{(1)}$ of rank $1$-degenerations
and observes that
the Picard group ${\rm Pic}(\A{g}^{(1)})\otimes {\bQ}$ is generated by two
divisors: $\lambda_1$ and the class $\delta$ of the boundary.
By Proposition \ref{OmegaisSym2} for the coarse moduli space
we know the canonical class.

\begin{proposition}
Let $U$ be the open subset of ${\cA}_g^{(1)}$ of 
semi-abelian varieties of torus rank $\leq 1$ 
with automorphism group $\{ \pm \}$. Then the canonical class of $U$
is given by $(g+1)\lambda_1 -\delta$.
\end{proposition}

Now there is an explicit effective divisor $N_0$ in ${\mathcal A}_g^{(1)}$,
namely the divisor of principally polarized abelian varieties
$(X,\Theta)$ that have a singular theta divisor. By a tricky calculation
in a $1$-dimensional family Mumford is able to calculate the divisor
class of $N_0$. 

\begin{theorem}
The divisor class of $\bar{N}_0$ is given by
$$
[\bar{N}_0]= \left( \frac{(g+1)!}{2}+g!\right) \lambda_1 
- \frac{(g+1)!}{12} \delta \, .
$$
\end{theorem}
The divisor $N_0$ is in general not irreducible and splits off the
divisor $\theta_0$ 
of principally polarized abelian varieties whose symmetric theta
divisor has a singularity at a point of order $2$. This divisor is
given by the vanishing of an even theta constant (Nullwert) 
and so this divisor class can be easily computed.
The divisor class of $\theta_0$ equals
$$
[\theta_0]= 2^{g-2}(2^g+1)\, \lambda_1-2^{2g-5}\, \delta \, .
$$

\begin{proof} (of Theorem \ref{KDgatleast7})
We have $[\bar{N}_0]=\theta_0+ 2 \, [R]$ with $R$ an effective divisor
given as
$$
[R]=\left( \frac{(g+1)!}{4} + \frac{g!}{2}-2^{g-3}(2^g+1)\right) \lambda_1 
-\left( \frac{(g+1)!}{24} -2^{2g-6}\right) \delta \, .
$$
For an expression $a\lambda_1-b \delta$ we call the ratio $a/b$ its slope.
If $R$ is effective with a slope $\leq$ slope of the expression for the
canonical class we know that the canonical class is ample.
From the formulas one sees that the slope of $\bar{N}_0$ is
$6+ 12/(g+1)$ and deduces that the inequality holds for $g\geq 7$.
\end{proof}

From the relatively easy equation for the class
$[\theta_0]$ one deduces that the even theta
constants provide a divisor with slope $8+2^{3-g}$, and this gives
that $\A{g}$ is of general type for $g\geq 8$; cf.\ \cite{Freitag:SMF}.

For $g\leq 5$ we know that $\A{g}$ is rational or unirational.
For $g=1$ and $2$ rationality was known in the 19th century; Katsylo proved
rationality for $g=3$, Clemens proved unirationality for $g=4$
and unirationality for $g=5$ was proved by several people (Mori-Mukai,
Donagi, Verra). But the case $g=6$ is still open.

For some results on the nef cone we refer to the survey 
\cite{Gr} of Grushevsky.

\section{Stratifications}
In positive characteristic the moduli space $\A{g}\otimes{\bF}_p$ 
possesses stratifications that can tell us quite a
lot about the moduli space; it is not clear what the characteristic zero
analogues of these stratifications are. This makes this moduli space
in some sense more accessible in characteristic $p$ than in characteristic 
zero, a fact that may sound counterintuitive to many.

The stratifications we are refering to are the Ekedahl-Oort stratification
and the Newton polygon stratification. Since the cycle classes are not known
for the latter one we shall stick to the Ekedahl-Oort
stratification, E-O for short. It was originally defined by Ekedahl and Oort
in terms of the group
scheme $X[p]$, the kernel of multiplication by $p$ for an abelian variety
in characteristic $p$. It has been studied intensively by Oort and many others,
see e.g. \cite{Oort:Stratification} and \cite{Oort:foliations}.

The alternative definition using degeneracy loci
of vector bundles was given in my paper \cite{vdG:Cycles} and worked out fully
in joint work with Ekedahl, see \cite{E-vdG:EO}. In this section
we shall write $\A{g}$ for $\A{g}\otimes {\bF}_p$. We start with the
universal principally polarized abelian variety $\pi:\X{g} \to \A{g}$.
For an abelian variey $X$ over a field $k$ 
the de Rham cohomology $H^1_{\rm dR}(X)$ is a vector space of rank $2g$.
We get by doing this in families a cohomology sheaf $\HDR(\X{g}/\A{g})$,
the hyperdirect image 
$R^1\pi_*({\mathcal O}_{\X{g}} \to \Omega^1_{\X{g}/\A{g}})$.
Because of the polarization it comes with a symplectic pairing $\langle \, , \, \rangle: \HDR \times \HDR \to {\mathcal O}_{\A{g}}$.
We have the Hodge filtration of $\HDR$:
$$
0 \to \pi_*(\Omega^1_{\X{g}/\A{g}}) \to \HDR \to R^1\pi_*{\mathcal O}_{\X{g}}
\to 0,
$$
where the first non-zero term is the Hodge bundle ${\bE}_g$. In characteristic
$p>0$ we have additionally two maps
$$
F: \X{g} \to \X{g}^{(p)}, \quad V: \X{g}^{(p)} \to \X{g},
$$
relative Frobenius and the Verschiebung 
satisfying $F\circ V= p \cdot {\rm id}_{\X{g}^{(p)}}$ and
$V\circ F = p \cdot {\rm id}_{\X{g}}$. 

Look at the simplest case  $g=1$. Since Frobenius is inseparable the kernel
$X[p]$ of multiplication by $p$ is not reduced and has either $1$ or $p$
physical points.
An elliptic curve in characteristic $p>0$
is called \emph{supersingular} 
if and only if $X[p]_{\rm red}=(0)$. Equivalently 
this means that $V$ is also inseparable, hence the kernel of $F$
and $V$ (for $\X{g}$ instead of $\X{g}^{(p)}$) coincide. There are
finitely many points on the $j$-line $\A{1}$ that correspond to the
supersingular elliptic curves. 

So in general 
we will compare the relative position of the kernel of $F$ and of $V$
inside $\HDR$. As it turns out, it is better to work with the flag
space $\F{g}$ of symplectic flags on $\HDR$; by this we mean that we
consider the space of flags $(E_i)_{i=1}^{2g}$ on $\HDR$ such that
${\rm rank} (E_i)=i$, $E_{g}= {\bE}$ and $E_{g-i}=E_j^{\bot}$. 
We can then introduce a second flag on $\HDR$, say $(D_i)_{i=1}^{2g}$
defined by setting
$$
D_g= \ker(V)=V^{-1}(0), \quad D_{g+i}=V^{-1}(E_i^{(p)}, 
$$
and complementing by
$$
D_{g-i}=D_{g+i}^{\bot} \qquad \text{for $i=1,\ldots,g$}.
$$
This is called the conjugate flag. For an abelian variety $X$ we define
the \emph{canonical flag} of $X$ as the coarsest flag that is stable
under $F$: if $G$ is a member, then also $F(G^{(p)})$; usually this will
not be a full flag. The conjugate flag $D$ is a refinement of the 
canonical flag.

So starting with one filtration we end up with two filtrations. We then
compare these two filtrations. In order to do this we need the Weyl group 
$W_g$ of the symplectic group ${\rm Sp}(2g,{\bZ})$. This group is isomorphic
to the semi-direct product 
$$
W_g \cong 
S_g \ltimes ({\bZ}/2{\bZ})^g 
= \{ \sigma \in S_{2g} : \sigma(i)+\sigma(2g+1-i)=2g+1, \, i=1,\ldots,g \} \, .
$$
As is well-known each element $w \in W_g$
has a \emph{length} $\ell(w)$. Let 
$$
W_g^{\prime}= \{ \sigma \in W_g: \sigma\{ 1,2, \ldots,g\} = \{1,2,\ldots,g\} \}
\cong S_g  \, .
$$
Now for every coset $aW_g^{\prime}$ there is a unique element $w$ of minimal
length in $aW_g^{\prime}$ such that every $w' \in aW_g^{\prime}$ can be written as $w'=wx$ with $\ell(w')=\ell(w)+\ell(x)$. 

There is a partial order on $W_g$, the Bruhat-Chevalley order:
$$
w_1 \geq w_2 \quad \text{if and only if} \quad r_{w_1}(i,j)\leq r_{w_2}(i,j)
\quad \text{for all $i,j$},
$$
where the function $r_w(i,j)$ is defined as
$$
r_w(i,j)= \# \{ n \leq i: w(n) \leq j \}.
$$
There are $2^g$ so-called \emph{final elements} in $W_g$: these are the
minimal elements of the cosets. In another context (see Section \ref{DF})
these are called 
Kostant representatives. An element $\sigma \in W_g$ is final 
if and only if $\sigma(i) < \sigma(j)$ for $i < j \leq g$. These final
elements correspond one-to-one to so-called final types: these are
increasing surjective maps
$$
\nu : \{0,1,\ldots,2g\} \to \{ 0,1,\ldots,2g\}
$$
satisfying $\nu(2g-i)=\nu(i)+(g-i)$ for $0 \leq i \leq g$. The bijection
is gotten by associating to $w \in W_g$ the final type $\nu_w$ with
$\nu_w(i)-i-r_w(g,i)$.

Now the stratification on our flag space $\F{g}$ is defined by
$$
(E_{\cdot},D_{\cdot}) \in U_w \iff \dim(E_i \cap D_j) \geq r_w(i,j) \, .
$$
There is also a scheme-theoretic definition: the two filtrations mean that
we have two sections $s,t$ of a $G/B$-bundle $T$ (with structure group $G$)
over the base with $G={\rm Sp}_{2g}$
and $B$ a Borel subgroup; locally (in the \'etale topology) we may assume that
$t$ is a trivializing section; then we can view $s$ as a map of the base
to $G/B$ and we simply take $U_w$ (resp.\ $\overline{U}_w$)
to be the inverse image of the $B$-orbit $BwB$ (resp.\ its closure).
This definition is independent of the chosen trivialization. 
We thus find global subschemes $U_w$ and $\bar{U}_w$ for
the $2^g (g!)$ elements of $W_g$.

It will turn out that for \emph{final} elements the projection map $\F{g}
\to \A{g}$ restricted to $U_w$ defines a finite morphism to its
image in $\A{g}$, and these will be the E-O strata. But the strata
on $\F{g}$ behave in a much better way. That is why we first study these
on $\F{g}$. We can extend the strata to strata over a compactification,
see \cite{E-vdG:EO}: the Hodge bundle extends and so does the de Rham
sheaf, namely as the logarithmic de Rham sheaf 
$R^1\pi_*(\Omega^{\cdot}_{\tilde{\X{g}}/\tilde{\A{g}}})$ and 
then we play the same game as above.

There is a smallest stratum: $\bar{U}_1$ associated
to the identity element of $W_g$. 

\begin{proposition}
Any irreducible component of any $\bar{U}_w$ in $\F{g}$ contains a point of $\bar{U}_1$.
\end{proposition}

The advantage of working in the flag space comes out clearly if we consider
the local structure of our strata. The idea is that locally our moduli 
space (flag space) looks like the space ${\rm FL}_g$ 
of complete symplectic flags in a
$2g$-dimensional symplective space. 

We need the notion of \emph{height $1$ maps} in charactistic $p>0$. A closed immersion $S \to S'$ is called a height $1$ map if $I^{(p)}=(0)$ with $I$
the ideal sheaf defining $S$ and $I^{(p)}$ the ideal generated by $p$-th
powers of elements. If $S={\rm Spec}(R)$ and $x: {\rm Spec}(k) \to S$
is a closed point then the height $1$ neighborhood is ${\rm Spec}(R/m^{(p)})$.
We can then introduce the notions of \emph{height $1$ isomorphic} and
\emph{height $1$ smooth} in an obvious way. 

Our basic result of \cite{E-vdG:EO} about our strata is now:

\begin{theorem}
Let $k$ be a perfect field. For every $k$-point $x$ of $\F{g}$ there exists
a $k$-point of ${\rm FL}_g$ such that their height $1$ neighborhoods
are isomorphic as stratified spaces.
\end{theorem}

Here the stratification on the flag space ${\rm FL}_g$ is the usual one
given by the Schubert cells.
The idea of the proof is to trivialize the de Rham bundle over a height $1$ 
neighborhood and then use infinitesimal Torelli for abelian varieties.
Our theorem has some immediate consequences.

\begin{corollary}\label{strataproperties}
\begin{enumerate}
\item Each stratum $U_w$ is smooth of dimension $\ell(w)$.
\item Each stratum $\bar{U}_w$ is Cohen-Macaulay, reduced and normal;
moreover $\bar{U}_w$ is the closure of $U_w$.
\item For $w$ a final element the projection $\F{g} \to \A{g}$ 
induces a finite \'etale map $U_w \to V_w$, with $V_w$ the image of $U_w$.
\end{enumerate}
\end{corollary}

We now descend from $\F{g}$ to $\A{g}$ to define the Ekedahl-Oort
stratification.

\begin{definition}
For a final element $w \in W_g$ the E-O stratum $V_w$ is defined to be the
image of $U_w$ under the projection $\F{g} \to \A{g}$.
\end{definition}

Over a E-O stratum $V_w \subset \A{g}$ the type of the canonical filtration
(i.e., the dimensions of the filtration steps that occur) are constant.

\begin{proposition}
The image of any $U_w$ under the projection $\F{g} \to \A{g}$ is a union
of strata $V_{v}$; the image of any $\bar{U}_w$ is a union of strata 
$\bar{V}_v$.
\end{proposition}

We also have an irreducibility result (cf.\ \cite{E-vdG:EO}).
Each final element $w \in W_g$ corresponds to a partition and a
Young diagram. 

\begin{theorem}
If $w \in W_g$ is a final element whose Young diagram $Y$ does not contain
all rows of length $i$ with
$\lceil (g+1)/2\rceil \leq i \leq g$ then $\bar{V}_w$ is irreducible and 
$U_w \to V_w$ is a connected \'etale cover.
\end{theorem}

Harashita proved in \cite{Harashita} that the other ones are in general
reducible (with exceptions maybe for small characteristics).

\smallskip
A stratification on a space is not worth much unless one knows the
cycle classes of the strata. In our case one can calculate these.

On the flag space the Chern classes $\lambda_i$ of the Hodge bundle
decompose in their roots:
$$
c_1(E_i)=l_1+\ldots + l_i\, .
$$
We then have
$$
c_1(D_{g+i})-c_1(D_{g+1-i})= p\, l_i.
$$

\begin{theorem}
The cycle classes of $\bar{U}_w$ are polynomials in the classes $l_i$
with coefficients that are polynomials in $p$.
\end{theorem}

We refer to \cite{E-vdG:EO} for an explicit formula. Using the analogues of
the maps $\pi_i$ of Section \ref{compactdual} and Lemma \ref{Gysin} 
we can calculate the cycle classes of the E-O strata on $\A{g}$.

Instead of giving a general formula we restrict to giving the formulas
for some important strata. For example, there are the $p$-rank strata
$$
V_f:=\{ [X] \in \A{g} : \# X[p](\bar{k})\leq p^f\}
$$
for $f=g,g-1,\ldots, 0$. Besides these there are the $a$-number strata
$$
T_a:=\{ [X] \in \A{g}: \dim_k {\rm Hom}(\alpha_p,X) \geq a \} \, .
$$
Recall that the $p$-rank $f(X)$ 
 of an abelian variety is $f$ if and only if
$\# X[p](\bar{k})=p^f$
and $0 \leq f \leq g$ with $f=g$ being the generic case. Similarly, the
$a$-number of $a(X)$  of $X$ is $\dim_k(\alpha_p, X)$ and this equals the rank of
$\ker (V) \cap \ker(F)$; so $0\leq a(X) \leq g$ with $a(X)=0$ being 
the generic case. The stratum $V_f$ has codimension $g-f$ while the stratum
$T_a$ has codimension $a(a+1)/2$. These codimensions were originally
calculated by Oort
and follow here easily from \ref{strataproperties}.

\begin{theorem} 
The cycle class of the $p$-rank stratum $V_f$ is given by
$$
[V_f] = (p-1)(p^2-1)\cdots (p^{g-f}-1) \lambda_{g-f}.
$$
\end{theorem}

For example, for $g=1$ the stratum $V_0$ is the stratum of supersingular
elliptic curves. We have $[V_0]= (p-1)\lambda_1$. By the cycle relation
$12 \lambda_1=\delta$ with $\delta$ the cycle of the boundary we find
$$
[V_0]= \frac{p-1}{12} \delta \, .
$$
Since the degenerate elliptic curve (rational nodal curve) has 
two automorphisms we find (using the stacky interpretation of our formula) 
the Deuring Mass
Formula
$$
\sum_{E/\bar{k} \text{\rm \, supersingular}} 
\frac{1}{\# {\rm Aut}_k(E)} = \frac{p-1}{24}
$$
for the number of supersingular elliptic curves in characteristic $p$.
One may view
all the formulas for the cycle classes as a generalization of the Deuring
Mass Formula.

The formulas for the $a$-number strata are given in \cite{vdG:Cycles} and
\cite{E-vdG:EO}. 
We have
\begin{theorem}
The cycle class of the locus $T_a$ of abelian varieties with $a$-number
$\geq a$ is given by
$$
\sum Q_{\beta}({\bE}^{(p)} \cdot Q_{\varrho(a)-\beta}({\bE}^{\vee}),
$$
where $Q_{\mu}$ is defined as in Section \ref{compactdual}
 and the sum is over all partitions
$\beta$ contained in  $\varrho(a)=\{a,a-1,\ldots,1\}$.
\end{theorem}
For example, the formula for the stratum $T_1$ is $p\lambda_1-\lambda_1$,
which fits since $a$-number $\geq 1$ means exactly that the $p$-rank
is $\leq g-1$. For $a=2$ the formula is
$[T_2]=(p-1)(p^2+1)(\lambda_1\lambda_2)-(p^3-1)2\lambda_3$. For $a=g$
the formula reads
$$
[T_g]=(p-1)(p^2+1)\cdots (p^g+(-1)^g) \lambda_1\lambda_2 \cdots \lambda_g.
$$
This stratum is of maximal codimension; this formula is visibly a
generalization of the
Deuring Mass Formula and counts the number of superspecial abelian varieties;
this number was first calculated by Ekedahl in \cite{Ekedahl:SS}. 

\smallskip

We formulate a corollary of our formulas for the $p$-rank strata.

\begin{corollary}
The Chern classes $\lambda_i$ of the Hodge bundle are represented
by effective  ${\bQ}$-classes.
\end{corollary}

Another nice aspect of our formulas is that when we specialize $p=0$
in our formulas, that are polynomials in the $l_i$ and $\lambda_i$
with coefficients that are polynomials in $p$, we get back the formulas
for the cycle classes of the Schubert cells both on the Grassmannian
and the flag space, see \cite{E-vdG:EO}.
\section{Complete subvarieties of $A_g$}
The existence of complete subvarieties of relative small codimension
can give us interesting information about a non-complete variety. 
In the case of the moduli space $\A{g}\otimes k$ (for some field $k$)
we know that $\lambda_1$
is an ample class and that $\lambda_1^{g(g-1)/2+1}$ vanishes. Since
for a complete subvariety $X$ of dimension $d$ we must have $\lambda_1^d|X
\neq 0$, this implies immediately a lower bound on the codimension.
\begin{theorem}
The minimum possible codimension of a complete subvariety of $\A{g}\otimes k$ 
is $g$.
\end{theorem}

This lower bound can be realized in positive characteristic as was
noted by Koblitz and Oort, see \cite{Koblitz}, \cite{Oort:SV}.
The idea is simple: a semi-abelian variety with a positive torus rank
has points of order $p$. So by requiring that our abelian varieties
have $p$-rank~$0$ we stay inside $\A{g}$ and this defines the required
complete variety.

\begin{theorem}
The moduli stack $\A{g}\otimes k$ with ${\rm char}(k)=p>0$
contains a complete substack of codimension $g$: the locus $V_0$ of
abelian varieties with $p$-rank zero.
\end{theorem}

A generalization is:

\begin{theorem}
The partial Satake compactification $\SatA{g}-\SatA{g-(t+1)}$ in
characteristic $p>0$ of
degenerations of torus rank $t$ contains a complete subvariety of
codimension $g-t$, namely the locus $V_{t}$ of $p$-rank $\leq t$.
\end{theorem}

In characteristic $0$ there is no such obvious complete subvariety
of codimension $g$, at least if $g\geq 3$. Oort conjectured in
\cite{Oort:CS} that it should not exist for $g\geq 3$. This was proved
by Keel and Sadun in \cite{Keel-Sadun}.

\begin{theorem} 
If $X \subset \A{g}({\bC})$ is a complete subvariety with the property
that $\lambda_i | X $ is trivial in cohomology for some $1 \leq i \leq g$
then $\dim X \leq i(i-1)/2$ with strict inequality if $i\geq 3$.
\end{theorem}

This theorem implies the following corollary.

\begin{corollary}
For $g \geq 3$ the moduli space $A_g \otimes {\bC}$ does not possess
a complete subvariety of codimension $g$.
\end{corollary}

So the question arises what the maximum dimension of a complete
subvariety of $\A{g}({\bC})$ is.

One might conclude that the analogue of $\A{g}({\bC})$ in positive
characteristic in some sense is rather the locus of principally polarized
abelian varieties with maximal $p$-rank ($=g$) than $\A{g}\otimes {\bF}_p$.
\section{Cohomology of local systems and relations to modular forms}
There is a close connection between the cohomology of moduli spaces
of abelian varieties and modular forms. This connection was discovered
in the 19th century and developed further in the work of Eichler, Shimura,
Kuga, Matsushima and many others. It has developed into a central 
theme involving the theory of automorphic representations and 
the Langlands philosophy.
We shall restrict here to just one aspect of this. This is work in 
progress that is being developed in joint work with Jonas Bergstr\"om
and Carel Faber.

Let us start with $g=1$. The space of cusp forms $S_{2k}$ of weight $2k$
on ${\rm SL}(2,{\bZ})$ has a cohomological interpretation. To describe
it we consider the universal elliptic curve $\pi: \X{1} \to \A{1}$
and let $V:=R^1\pi_* {\bQ}$ be the local system of rank $2$ with as
fibre over $[X]$ the cohomology $H^1(X,{\bQ})$ of the elliptic curve $X$

From $V$ we can construct other local systems: define for $a\geq 1$
$$
V_a:= {\rm Sym}^a(V).
$$
This is a local system of rank $a+1$. There is also the $l$-adic analogue
$V^{(l)}=R^1\pi_*{\bQ}_l$ 
for $l$-adic \'etale cohomology and its variants $V_a^{(l)}$.
The basic result of Eichler and Shimura says that for $a>0$ and $a$ even
there is an isomorphism
$$
H^1_c(\A{1}\otimes {\bC},V_a \otimes {\bC}) 
= S_{a+2}\oplus \bar{S}_{a+2} \oplus  {\bC}.
$$
So we might say that as a mixed Hodge module the compactly supported
cohomology of the local system $V_a\otimes {\bC}$ equals 
$S_{a+2}\oplus \bar{S}_{a+2} \oplus {\bC}$.

But this identity can be stretched
further. The left hand side may be replaced by other flavors of cohomology,
for example by $l$-adic \'etale cohomology $H^1(\A{1}\otimes \bar{\bQ}, V_a^{(l)})$
that comes with a natural Galois action of ${\rm Gal}(\bar{\bQ}/{\bQ})$.
In view of this, we replace the left hand side by an Euler characteristic
$$
e_c(\A{1},V_a):= \sum_{i=0}^2 [H^i_c(\A{1},V_a)],
$$
where now the cohomomology groups of compactly supported cohomology are
to be interpreted in an appropriate Grothendieck group, e.g.\ of mixed
Hodge structures when we consider complex cohomology $H^*_c(\A{1}\otimes {\bC},
V_a\otimes {\bC})$ with its mixed Hodge
structure, or in the Grothendieck group of Galois representations when
we consider compactly supported \'etale $l$-adic cohomology 
$H^*_{c,et}(\A{1}\otimes \bar{\bQ},V_a^{(l)})$.
On the other hand, for the right hand side Scholl defined in \cite{Scholl}
a Chow motive
$S[2k]$ associated to the space of cusp forms $S_{2k}$ for $k>1$. 
Then a sophisticated form of the Eichler-Shimura isomorphism asserts
that we have an isomorphism
$$
e_c(\A{1},V_a)= -S[a+2]-1 \qquad \text{$a\geq 2$ even}\, .
$$

We have a natural algebra of operators, the Hecke
operators,  acting on 
the spaces of cusp forms, but also on the cohomology since the Hecke
operators are defined by correspondences. 
Then the isomorphism above is compatible with the action of the Hecke
operators.

But our moduli space $\A{1}$ is defined over ${\bZ}$. We thus can study
the cohomology over ${\bQ}$ by looking at the fibres $\A{1} \otimes {\bF}_p$
and the corresponding local systems $V_a^{(l)} \otimes {\bF}_p$ (for $l\neq p$)
by using comparison theorems. Now in characteristic $p$ the Hecke operator
is defined by the correspondence $X_0(p)$ 
of (cyclic) $p$-isogenies $\phi: X \to X'$ between
elliptic curves (i.e.\ we require $\deg \phi =p$); it allows maps
$q_i: X_0(p) \to \A{1}$ ($i=1,2$) 
by sending $\phi$ to its source $X$ and target $X'$.
In characteristic $p$ the correspondence decomposes into two components
$$
X_0(p)\otimes {\bF}_p = F_p+F_p^t, \qquad \text{(the congruence relation)}
$$
where $F_p$ is the correspondence $X \mapsto X^{(p)}$ and $F_p^t$ its transpose.
This follows since for such a $p$-isogeny we have that $X'\cong X^{(p)}$ or
$X\cong (X')^{(p)}$. 
This implies a relation between the Hecke operator $T(p)$ and the action
of Frobenius on $H^1(\A{1}\otimes \bar{\bF}_p, V_a^{(\ell)})$; 
The result is then a relation between the action of Frobenius on 
\'etale cohomology
and the action of a Hecke operator on the space of cusp forms 
(\cite{Deligne:FM}, Prop.\ 4.8)
$$
{\rm Tr}(F_p,H^1_c(\A{1}\otimes \overline{\bF}_p,V_a))
= {\rm Tr}(T(p),S_{a+2})+1\, .
$$
We can calculate the traces of the
Frobenii by counting points over finite fields. In fact, if we make a
list of all elliptic curves over ${\bF}_p$ up to isomorphism over ${\bF}_p$
and we calculate the eigenvalues $\alpha_X,\bar{\alpha}_X$ 
of $F_p$ acting on the \'etale 
cohomology $H^1(X,{\bQ}_l)$ then we can calculate the trace of $F_p$
on the cohomology $H^1_c(\A{1}\otimes \bar{\bF}_p,V_a^{(\ell)})$ 
by summing the expression
$$
\frac{\alpha_X^a+ \alpha_X^{a-1}\bar{\alpha}_X+ \ldots + \bar{\alpha}_X^a}{
\# {\rm Aut}_{{\bF}_p}(X)}
$$
over all $X$ in our list. 
So by counting elliptic curves over a finite field ${\bF}_p$ we can calculate
the trace of the Hecke operator $T(p)$ on $S_{a+2}$; in fact, once we have
our list of elliptic curves over ${\bF}_p$ together with the 
eigenvalues $\alpha_X, \bar{\alpha}_X$ and the order of
their automorphism group, we can calculate the trace of $T(p)$ on
the space of cusp forms  $S_{a+2}$
for \emph{all} $a$.

The term $-1$ in the Eichler-Shimura identity can be interpreted as
coming from the kernel
$$
-1= \sum (-1)^i [\ker H^i_c(\A{1},V_a) \to H^i(\A{1},V_a)].
$$
So to avoid this little nuisance we might replace the compactly supported
cohomology by the image of compactly supported cohomology in the usual
cohomology, i.e.\ define the \emph{interior} cohomology $H^i_!$
as the image of compactly supported cohomology in the usual cohomology.
Then the result reads
$$
e_!(\A{1},V_a)= -S[a+2] \qquad \text{for $a>2$ even}.
$$

Some words about the history of the Eichler-Shimura result may be in order
here.
Around 1954 Eichler showed 
(see \cite{Eichler1954}) that for some congruence subgroup $\Gamma$ of
${\rm SL}(2,{\bZ})$ the $p$-part of the zeta function of the
corresponding modular curve $X_{\Gamma}$ in characteristic $p$ is given by the
Hecke polynomial for the Hecke operator $T(p)$ acting on the space
of cusp forms of weight $2$ for $\Gamma$. This was generalized to
some other groups by Shimura. M.\ Sato observed in 1962 that
by combining the Eichler-Selberg trace formula for modular forms
with the congruence relation (expressing the Hecke correspondence
in terms of the Frobenius correspondence and its transpose) one
could extend Eichler's results by expressing the Hecke polynomials
in terms of the zeta functions of ${\mathcal M}_{1,n} \otimes {\bF}_p$,
except for problems due to the non-completeness of the moduli spaces.
A little later Kuga and Shimura showed that Sato's idea worked for 
compact quotients
of the upper half plane (parametrizing abelian surfaces).
Ihara then made Sato's idea reality in 1967 by combining the
Eichler-Selberg trace formula with results of Deuring and Hasse
on zeta functions of elliptic curves, cf.\ \cite{Ihara}, where
one also finds references to the history of this problem.
A year later Deligne solved in \cite{Deligne:FM}
the problems posed by the non-completeness
of the moduli spaces. Finally Scholl proved the existence of the motive
$S[k]$ for even $k$ in 1990. A different construction of this motive
was given by Consani and Faber in \cite{C-F}.

That the approach sketched above for calculating traces of Hecke operators
by counting over finite fields is not used commonly,
is due to the fact that we
have an explicit formula for the traces of the Hecke operators, the
Eichler-Selberg trace formula, cf.\ \cite{Zagier}. But for higher genus $g$,
i.e.\ for modular forms on the symplectic group ${\rm Sp}(2g,{\bZ})$
with $g\geq 2$,
no analogue of the trace formula for Siegel modular forms is known.
Moreover, a closer inspection reveals that vector-valued Siegel modular
forms are the good analogue for higher $g$ of the modular forms on
${\rm SL}(2,{\bZ})$ (rather than classical Siegel modular forms only). 
This suggests to try the analogue of the point counting
method for genus $2$ and higher. That is what Carel Faber and I did
for ${\rm Sp}(4,{\bZ})$ and in joint work with Bergstr\"om extended
to genus $2$ and level $2$ and also to genus $3$, see \cite{FvdG:CR,BFG,BFG:g=3}.

There are alternative methods that we should point out.
In general one tries to compare a trace formula of
Selberg type (Arthur trace formula) with the Grothendieck-Lefschetz
fixed point formula.
There is a topological trace formula for the trace 
of Hecke operators acting on the compactly supported cohomology, 
see for example \cite{Harder}, esp.\ the letter 
to Goresky and MacPherson there.
Laumon gives a spectral decomposition of the cohomology of 
local systems for $g=2$ (for the trivial one in \cite{Laumon1} and in general
in \cite{Laumon2}); cf.\ also the work of Kottwitz in general.
We also refer to Sophie Morel's book \cite{SM}.
But though these methods in principle could lead to explicit results 
on Siegel modular forms, as far as I know it has not yet done that.

So start with the universal principally polarized abelian variety
$\pi: \X{g} \to \A{g}$ and form the local system $V:=R^1\pi_* {\bQ}$
and its $l$-adic counterpart $R^1\pi_* {\bQ}_{l}$ for \'etale
cohomology. 
By abuse of notation we will write simply $V$. 
We consider $\pi$  as a morphism of stacks.
For any irreducible representation $\lambda$ of ${\rm Sp}_{2g}$
we can construct a corresponding local system $V_{\lambda}$; it is obtained
by applying a Schur functor, cf.\ e.g.\ \cite{F-H}. 
So if we denote $\lambda$ by its highest
weight $\lambda_1\geq \lambda_2 \geq \cdots \geq \lambda_g$ then
$V=V_{1,0,\ldots,0}$ and ${\rm Sym}^a(V)=V_{a,0,\ldots,0}$.

We now consider
$$
e_c(\A{g},V_{\lambda}):= \sum_i (-1)^i [H^i_c(\A{g},V_{\lambda})],
$$
where as before the brackets indicate that we consider this in the
Grothendieck group of the appropriate category (mixed Hodge modules,
Galois representations). It is a basic result of Faltings \cite{F-C,F}
that $H^i(\A{g}\otimes {\bC},V_{\lambda} \otimes {\bC})$ and 
$H^i_c(\A{g}\otimes {\bC},V_{\lambda}\otimes {\bC})$ 
carry a mixed Hodge structure. Moreover, the interior cohomology $H^i_!$
carries a pure Hodge structure with the weights equal to the $2^g$
sums of any of the subsets of 
$\{ \lambda_1+g,\lambda_2+g-1,\ldots,\lambda_g+1\}$. He also shows that for 
regular $\lambda$, that is, when $\lambda_1> \lambda_2 \cdots
>\lambda_g >0$, the cohomology $H^i_!$ vanishes when $i\neq g(g+1)/2$.

These results of Faltings are the analogue for the non-compact case of results
of Matsushima and Murakami in \cite{MM}; they could use Hodge theory
to deduce the vanishing of cohomology groups and the decompositions
indexed by pairs of elements in the Weyl group.

For $g=2$ a local system $V_{\lambda}$  is specified by giving $\lambda=(a,b)$
with $a\geq b \geq 0$. Then the Hodge filtration is
$$
F^{a+b+3} \subseteq F^{a+2} \subseteq F^{b+1} \subseteq F^0=
H^3_!(\A{2}\otimes{\bC},V_{\lambda}\otimes {\bC}).
$$
Moreover, there is an identification of the first step in the Hodge filtration
$$
F^{a+b+3} \cong S_{a-b,b+3},
$$
with the space $S_{a-b,b+3}$
of Siegel modular cusp forms whose weight $\rho$
is the representation ${\rm Sym}^{a-b}{\rm St} \otimes \det({\rm St})^{b+3}$ 
with ${\rm St}$ the standard representation of ${\rm GL}(2,{\bC})$.

Again, as for $g=1$, we have an algebra of Hecke operators induced by geometric
correspondences and it acts both on the cohomology and the modular forms
compatible with the isomorphism. 

Given our general ignorance of vector-valued Siegel modular forms for
$g>1$ the obvious question at this point is whether we can mimic the 
approach sketched above for calculating the traces of the Hecke
operators by counting over finite fields.

Note that by Torelli we have morphism $\M{2}\to \A{2}=\M{2}^{\rm ct}$,
where $\M{2}^{ct}$ is the moduli space of curves of genus $2$
of compact type. This
means that we have to include besides the smooth curves of genus $2$
the stable curves of genus $2$ that are a union of two elliptic curves.

Can we use this to calculate the traces of the Hecke operator $T(p)$
on the spaces of vector-valued Siegel modular forms?

Two problems arise. The first is the so-called Eisenstein cohomology. This
is 
$$
e_{\rm Eis}(\A{2}, V_{a,b}):= \sum (-1)^i (\ker H^i_c \to H^i)(\A{2},V_{a,b}).
$$
In the case of genus $1$ this expression was equal to the innocent $-1$,
but for higher $g$ this is a more complicated term. The second problem
is the \emph{endoscopy}: the terms in the Hodge filtration that do 
not see the first and last step of the Hodge filtration. 
For torsion-free groups there is work by 
Schwermer (see \cite{Schwermer})
on the Eisenstein cohomology; cf.\ also the work of Harder \cite{Harder}.
And there is an extensive literature on endoscopy.
But explicit formulas were not available.

On the basis of numerical calculations Carel Faber and I guessed 
in \cite{FvdG:CR} a formula
for the Eisenstein cohomology. I was able to prove this formula in 
\cite{vdG:EC} for regular $\lambda$. 
We also made a guess for the endoscopic term.
Putting this together we get the following conjectural formula
for $(a,b)\neq (0,0)$.

\begin{conjecture}\label{g=2conj}
The trace of the Hecke operator $T(p)$ on the space $S_{a-b,b+3}$ of 
cusp forms on ${\rm Sp}(4,{\bZ})$ equals 
$$
-{\rm Tr}(F_p,e_c(\A{2}\otimes {\bF}_p,V_{a,b}))+
{\rm Tr}(F_p,e_{2,\rm extra}(a,b)),
$$
where the term $e_{2,\rm extra}(a,b)$ is defined as
$$
s_{a-b+2}-s_{a+b+4}(S[a-b+2]+1) {\bL}^{b+1} +
\begin{cases} S[b+2]+1 & a \, \text{even},\\ -S[a+3] & a \, \text{odd}, \\
\end{cases}
$$
and
$s_n= \dim S_n({\rm SL}(2,{\bZ}))$ is the dimension of the space of 
cusp forms on ${\rm SL}(2,{\bZ})$ and ${\bL}=h^2({\bP}^1)$ 
is the Lefschetz motive. 
\end{conjecture}

If $a>b=0$ or $a=b>0$ then one should put $s_2=-1$ and $S[2]=-1-{\bL}$ in
the formula.  

In the case of regular local systems (i.e.\ $a>b>0$) our conjecture can be 
deduced from results of Weissauer as he shows in his preprint
\cite{Weissauer:THO}. In the case of local system of non-regular
highest weight the conjecture is still open.

We have counted the curves of genus $2$ of compact type for all primes
$p \leq 37$. This implies that we can calculate the traces of the
Hecke operator $T(p)$ for $p\leq 37$  on the space of cusp forms $S_{a-b,b+3}$
{\sl for all} $a>b>0$, and assuming the conjecture also for $a>b=0$ and $a=b>0$.

Our results are in accordance with results on the numerical Euler
characteristic $\sum (-1)^i \dim H^i_c(\A{2}, V_{a,b})$ due to Getzler
(\cite{Getzler}) and with the
dimension formula for the space of cusp form $S_{a-b,b+3}$ due to Tsushima
(\cite{Tsushima}).

We give an example. For $(a,b)=(11,5)$ we have
$$
e_c(\A{2},V_{11,5})= -{\bL}^6 -S[6,8]
$$
with $S[6,8]$ the hypothetical motive associated to the ($1$-dimensional)
space of cusp forms $S_{6,8}$.
We list a few eigenvalues $\lambda(p)$ and $\lambda(p^2)$
of the Hecke operators. This allows us to give 
the characteristic polynomial of Frobenius (and the Euler factor 
of the spinor $L$-function of the Siegel modular form)
$$
1-\lambda(p)X+(\lambda(p)^2-\lambda(p^2)-p^{a+b+2})X^2-\lambda(p)p^{a+b+3}X^3+
p^{2(a+b+3)} X^4
$$
and its slopes.

\smallskip
\vbox{
\bigskip\centerline{\def\quad{\hskip 0.6em\relax}
\def\quod{\hskip 0.5em\relax }
\vbox{\offinterlineskip
\hrule
\halign{&\vrule#&\strut\quod\hfil#\quad\cr
height2pt&\omit&&\omit&&\omit&&\omit&\cr
&$p$&&$\lambda(p)$&&$\lambda(p^2)$&&{\rm slopes}&\cr
height2pt&\omit&&\omit&&\omit&&\omit&\cr
\noalign{\hrule}
height2pt&\omit&&\omit&&\omit&&\omit&\cr
&$2$&&$0$&&$-57344$&&$13/2,25/2$&\cr
&$3$&&$-27000$&&$143765361$&&$3,7,12,16$&\cr
&$5$&&$2843100$&&$-7734928874375$&&$2,7,12,17$&\cr
&$7$&&$-107822000$&&$4057621173384801$&&$0,6,13,19$&\cr
height2pt&\omit&&\omit&&\omit&&\omit&\cr
} \hrule}
}}

Another indication that the computer counts of curves of genus $2$
are correct comes from a conjecture of Harder about congruences
between Hecke eigenvalues of cusp forms for genus $1$ and genus $2$.
Harder had the idea that there should be such congruences many
years ago (cf.\ \cite{Harder}), 
but our results motivated him to make his conjecture 
precise and explicit. He conjectured that if a (not too small,
or better an ordinary)
prime  $\ell$ divides a critical
value $s$ of the $L$-function of an eigenform on ${\rm SL}(2,{\bZ})$
of weight $r$ then there should be a vector-valued 
Siegel modular form of prescribed weight depending on $s$ and $r$
and a congruence modulo $\ell$ between the eigenvalues under the Hecke
operator $T(p)$ for $f$ and $F$. 
We refer to \cite{BGHZ}, the papers
by Harder \cite{Harder:123}  and van der Geer \cite{vdG:SMF} 
there, for an account of this fascinating story.
These congruences generalize the famous congruence
$$
\tau(n) \equiv p^{11}+1 \, (\bmod \, 691)
$$
for the Hecke eigenvalues  $\tau(p)$ ($p$ a prime) 
of the modular form $\Delta=\sum_{n>0} \tau(n) q^n$
of weight $12$ on ${\rm SL}(2,{\bZ})$.
One example of such a congruence is the congruence
$$
\lambda(p) \equiv p^8+a(p)+p^{13} \, (\bmod \, 41)
$$
where $f=\sum a(n) q^n$ is the normalized ($a(1)=1$) cusp form of
weight $22$ on ${\rm SL}(2,{\bZ})$ and the $\lambda(p)$ are the Hecke
eigenvalues of the genus $2$ Siegel cusp form $F \in S_{4,10}$.
We checked this congruence for all primes $p$ with $p\leq 37$.
For example, $a(37)=22191429912035222$ and $\lambda(37)=
11555498201265580$.

In joint work with Bergstr\"om and Faber  we extended this to
level~$2$. One considers the moduli space $\A{2}[2]$ of 
principally polarized abelian surfaces of level~2. This moduli
space contains as a dense open subset the moduli space
$\M{2}[w^6]$ of curves of genus~$2$ together with six Weierstrass
points. It comes with an action of $S_6\cong {\rm Sp}(4,{\bZ}/2{\bZ})$.
We formulated an analogue of conjecture \ref{g=2conj} for level~$2$.
Assuming this conjecture we can calculate the traces of the Hecke
operators $T(p)$ for $p\leq 37$ for the spaces of cusp forms of all
level~$2$. Using these numerical data we could observe lifings
from genus~$1$ to genus~$2$ and could make precise conjectures
about such liftings; and again we could predict and verify numerically
congruences  between genus~$1$ and genus~$2$ eigenforms.

We give an example. For $(a,b)=(4,2)$ we find
$$
e_c(\A{2}[2],V_{4,2})= -45 {\bL}^3+45 -S[\Gamma_2[2],(2,5)]
$$
and assuming our conjecture 
we can calculate the traces of the Hecke operators on the space
of cusp forms of weight $(2,5)$ on the level $2$ congruence
subgroup $\Gamma_2[2]$ of ${\rm Sp}(4,{\bZ})$; 
this space is a representation of
type $[2^2,1^2]$ for the symmetric group $S_6$ and is generated by
one Siegel modular form; for this Siegel modular form we have
the eigenvalue 
$\lambda(23)=-323440$ for $T(23)$.

It is natural to ask how the story continues for genus $3$. The first remark is
that the Torelli map $\M{3} \to \A{3}$ is a morphism of degree $2$
in the sense of stacks. This is due to the fact that every principally
polarized abelian variety has a non-trivial automorphism $-{\rm id}$,
while the general curve of genus $3$ has a trivial automorphism group.
This has as a consequence that for local systems $V_{a,b,c}$ with $a+b+c$
odd the cohomology on $\A{3}$ vanishes, but on $\M{3}$ it need not; and in
fact in general it does not. 

In joint work with Bergstr\"om and Faber we managed to formulate
an analogue of \ref{g=2conj} for genus $3$, see \cite{BFG:g=3}. 
Assuming the conjecture
we are able to calculate the traces of Hecke operators $T(p)$ on the
space of cusp forms $S_{a-b,b-c,c+4}$ for
all primes $p$ for which we did the counting (at least $p\leq 19$).
Again the numerically data fit with calculations of dimensions of 
spaces of cusp forms and numerical Euler characteristics. (These numerical
Euler characteristics were calculated in \cite{B-vdG:EC}.) And again
we could observe congruences between eigenvalues for $T(p)$ for 
$g=3$ eigenforms and those of genus $1$ and $2$. 
We give two examples:
$$
e_c(\A{3},V_{10,4,0})= -{\bL}^7+{\bL}+S[6,8],
$$
the same $S[6,8]$ for genus $2$ we met above. And for example
$$
e_c(\A{3},V_{8,4,4})= -S[12] \, {\bL}^6+S[12] +S[4,0,8],
$$
where genuine Siegel modular forms (of weight $(4,0,8)$)
of genus $3$ do occurr.
We refer to \cite{BFG:g=3} for the details and to the Chapter by 
Faber and Pandharipande for the cohomology of local systems
and modular forms on the moduli spaces of curves.

\smallskip
\noindent
{\sl Acknowledgement}
The author thanks Jonas Bergstr\"om and Carel Faber
for the many enlightening discussions we had on topics dealt with in
this survey. Thanks are due to T.\ Katsura for inviting me to Japan where
I found the time to finish this survey. I am also greatly indebted to 
Torsten Ekedahl from whom I learned so much. I was a great shock to
hear that he passed away; his gentle and generous personality
and sharp intellect will be deeply missed.

\end{document}